\documentclass{amsart}

\usepackage{amssymb}
\usepackage{amsmath}
\usepackage{amssymb}
\usepackage{latexsym}
\usepackage{euscript}
\usepackage{enumerate}
\usepackage{mathdots}

\usepackage{xcolor}

{\left[\begin{smallmatrix}}
{\end{smallmatrix}\right]}%

\newcommand{\setdef}[2]{\left\{ #1 :\, \vphantom{#1} #2 \right\}}

   \def\sH{{\mathfrak H}}   
   \def\sK{{\mathfrak K}}   
\def\sM{{\mathfrak M}}      
      \def\sR{{\mathfrak R}}
      
\def\sV{{\mathfrak V}}   \def\sW{{\mathfrak W}}   \def\sX{{\mathfrak X}}
   \def\sZ{{\mathfrak Z}}

      \def\dC{{\mathbb C}}

   \def\dN{{\mathbb N}}

\def\bm\chi{\mbox{\boldmath$\chi$}}

\def\la{\lambda}

\def\si{\sigma}

\let\xker=\ker \def\ker{{\xker\,}}

\DeclareMathOperator{\ran}{ran}
\DeclareMathOperator{\dom}{dom}
\DeclareMathOperator{\mul}{mul}

\DeclareMathOperator{\spn}{span}

\def\hplus {{\, \widehat + \,}}

\newtheorem{theorem}{Theorem}[section]
\newtheorem{proposition}[theorem]{Proposition}
\newtheorem{corollary}[theorem]{Corollary}
\newtheorem{lemma}[theorem]{Lemma}
\newtheorem{definition}[theorem]{Definition}
\theoremstyle{definition}

\newtheorem{remark}[theorem]{Remark}

\numberwithin{equation}{section}

\makeatletter
\newcommand{\superimpose}[2]{%
  {\ooalign{$#1\@firstoftwo#2$\cr\hfil$#1\@secondoftwo#2$\hfil\cr}}}
\makeatother

\begin{document}

\title[Jordan-like decomposition for relations]
{A Jordan-like decomposition for linear relations in finite-dimensional spaces}

\author{Thomas Berger}
\address{Universit\"at  Paderborn\\
Institut f\"{u}r Mathematik\\
Warburger Str.\ 100\\
33098 Paderborn\\ Germany}
\email{thomas.berger{\scriptsize @}math.upb.de}
\author{Henk de Snoo}
\address{Bernoulli Institute for Mathematics, Computer Science
and Artificial Intelligence\\
University of Groningen \\
P.O.\ Box 407, 9700 AK Groningen \\
Netherlands}
\email{h.s.v.de.snoo{\scriptsize @}rug.nl}
\author{Carsten Trunk}
\address{Institut f\"{u}r Mathematik \\
Technische Universit\"{a}t Ilmenau\\
Weimarer Stra{\ss}e~25, 98693~Ilmenau \\
Germany}
\email{carsten.trunk{\scriptsize @}tu-ilmenau.de}
\author{Henrik Winkler}
\address{Institut f\"{u}r Mathematik \\
Technische Universit\"{a}t Ilmenau\\
Weimarer Stra{\ss}e~25, 98693~Ilmenau \\
Germany}
\email{henrik.winkler{\scriptsize @}tu-ilmenau.de}


%
%
%
%
%
%
%

\begin{abstract}
A square matrix $A$ has the usual Jordan canonical form that
describes the structure of $A$ via eigenvalues and the
corresponding Jordan blocks.
If $A$ is a linear relation in a finite-dimensional linear space $\sH$
(i.e., $A$  is a linear subspace of $\sH \times \sH$ and can be considered
as a multivalued linear operator), then there is a richer structure.
In addition to the classical Jordan chains (interpreted
in the Cartesian product $\sH \times \sH$),
there occur three more classes of chains: chains
starting at zero (the chains for the eigenvalue infinity),
chains starting at zero and also ending at zero
(the singular chains), and chains 
with linearly independent entries
(the shift chains). These four types of chains give rise to a direct sum decomposition (a Jordan-like decomposition)
of the linear relation $A$.
In this decomposition there is a completely singular part that has
the extended complex plane as eigenvalues; a usual Jordan part that corresponds to the finite
proper eigenvalues; a Jordan part that corresponds to the eigenvalue $\infty$;
and a multishift, i.e., a part that has no eigenvalues at all. Furthermore, the Jordan-like decomposition exhibits a certain uniqueness, closing a gap in earlier results.
The presentation is purely algebraic, only the structure of linear spaces is used. Moreover, the presentation has a uniform character:
each of the above types is constructed via an appropriately chosen
sequence of quotient spaces. The dimensions of the spaces are the Weyr characteristics, which uniquely determine the Jordan-like decomposition of the linear relation.
\end{abstract}

\maketitle

\section{Introduction}\label{Sec:intro}

{ Let $\sH$ be a finite-dimensional linear space over $\dC$ and let $A$ be
a linear operator in $\sH$ with $\dom A=\sH$,
i.e., $A$ is defined everywhere and admits a representation as a matrix.
Then there is at least one eigenvalue  $\lambda \in \dC$
and to each eigenvalue belong chains of linearly independent vectors
$x_1,\ldots,x_n$, the so-called Jordan chains
\begin{equation}\label{pril}
(A-\lambda)x_n=x_{n-1}, \,(A-\lambda)x_{n-1}=x_{n-2}, \,\dots \,,\,(A-\lambda)x_1=0.
\end{equation}
The  Jordan canonical form of a matrix
offers a decomposition in terms of these Jordan chains.
For each   $\lambda \in \dC$ define
the sequence of quotient spaces
\begin{equation}\label{UnBeso-}
 \ker (A-\lambda), \,  \frac{\ker  (A-\lambda)^2}{\ker (A-\lambda)}, \,
 \frac{\ker  (A-\lambda)^3}{\ker (A-\lambda)^2}, \, \cdots
\end{equation}
and the corresponding Weyr characteristic by the sequence
\begin{equation}\label{UnBeso}
\dim \ker (A-\lambda),\,  \dim\frac{\ker  (A-\lambda)^2}{\ker (A-\lambda)},
\, \dim\frac{\ker  (A-\lambda)^3}{\ker (A-\lambda)^2}, \, \cdots .
\end{equation}
Then the Jordan canonical form is the unique representative of the equivalence class of~$A$ with respect to similarity, and it is uniquely determined by the Weyr characteristic. Furthermore, two matrices are similar if and only if their Weyr characteristics coincide.

For a recent treatment of the Weyr characteristic of matrices
and a historical discussion see~\cite{S99} (and also \cite{S15}).
The above condition that $\dom A=\sH$ ensures the existence
of at least one eigenvalue. If this condition is not satisfied
or if $A$ is a linear relation (multivalued operator),
then new phenomena may occur.
For instance, it may happen that $A$ has no eigenvalues at all and the
Jordan canonical form breaks down.

The purpose of the present note is to derive a general decomposition
for a linear relation $A$ in a finite-dimensional space $\sH$ over $\dC$, i.e., $A$ is a subspace of $\sH\times \sH$. Linear relations in linear spaces date back to \cite{Arens}, see also~\cite{BHS,BSTW,Cross,SandDeSn07}. Compared to linear operators, linear relations may have an eigenvalue $\infty$ with its own Jordan chains.
However, in the context of a linear relation $A$  there is also a new feature: it may happen
that the usual point spectrum $\sigma_{p}(A)$ is equal to the extended complex plane.
 For instance, this is the case when there exists a nontrivial element in
$\ker A \cap \mul A$.
By splitting off the so-called completely singular part~$A_S$ of $A$, there remains
the proper point spectrum $\sigma_{\pi}(A)$ of $A$, consisting of finitely many points in
$\dC \cup \{\infty\}$; see \cite{BSTW}. The main result is the following Jordan-like direct sum decomposition
of the linear relation $A$:
\begin{equation}\label{LanadelRey+}
A = A_S \oplus J_{\la_1}(A)\oplus \cdots\oplus J_{\la_l}(A)\oplus J_\infty(A) \oplus A_M,
\end{equation}
where $J_{\lambda}(A)$ stands for the Jordan part corresponding to
$\lambda \in \sigma_{\pi}(A) = \{\lambda_1,\ldots,\lambda_l\} \cup \{\infty\}$, and $A_{M}$ is a multishift, i.e., a linear operator
without eigenvalues; cf. Theorem~\ref{splitit} for the precise meaning of~\eqref{LanadelRey+}.
For each component in \eqref{LanadelRey+},
there is, parallel to the case of matrices in \eqref{UnBeso},
a suitably chosen sequence of quotient spaces with its own Weyr characteristic. The collection of the Weyr characteristics of each part defines
the \emph{Weyr characteristic} of the linear relation $A$;
cf. Definition~\ref{wweyr}.
It is a complete set of invariants which justifies to view \eqref{LanadelRey+}
as a Jordan-like decomposition:
{\it The decomposition  \eqref{LanadelRey+} is uniquely determined by the Weyr characteristic and it is the unique representative of the equivalence class with respect to strict similarity}; cf. Section~\ref{Sec:Decomp}.

In a sense, the present paper can be seen as a completion of the results
in~\cite{SandDeSn05} with a purely linear algebra approach;
see also \cite{BSTW}.
In fact, the decomposition derived in~\cite{SandDeSn05} exhibits a certain non-uniqueness; cf.\ Section~\ref{Sec:Decomp}. This is resolved by utilizing the concept of a reducing sum decomposition, which is intrinsically unique; cf.\ Section~\ref{Sec:preli}. To achieve such a decomposition,  it is
required to use a completely different construction of the subrelations.

The present paper is organized as follows: The necessary notions of root spaces, chains and reducing sum decompositions for linear relations are recalled in Section~\ref{Sec:preli}.
The construction of each part in the decomposition~\eqref{LanadelRey+} of the linear relation $A$ follows a uniform pattern: The discussion of the four kinds of sequences of quotient spaces, and the resulting chain structure is the content of
Sections~\ref{Sec:SingJordanPart},~\ref{sec3}, and~\ref{Sec:Multishift};
see \eqref{tunesandI}, \eqref{vvv00}, \eqref{vv0}, \eqref{grijp0}, and~\eqref{eq:Dk}.
In Section~\ref{Sec:Decomp} the main decomposition results of the paper are collected
and explained in terms of the Weyr characteristic. Section \ref{Sec:Decomp} also contains a brief discussion of
related literature.

The above characterization of linear relations via
their Weyr characteristics is new and allows for a variety of applications,
in particular to linear matrix pencils. With any linear pencil one may associate a kernel and a range representation, which are two different linear relations.
The relationship between the reducing sum decomposition~\eqref{LanadelRey+} of these linear relations
and the Kronecker canonical form of the original matrix pencil is of great interest.
}

\section{Preliminaries}\label{Sec:preli}

\subsection{Linear relations}

A linear relation $A$ in a finite dimensional linear space $\sH$ is a subspace of $\sH\times \sH$. In the following a brief review of the usual notions in the context of linear relations is given:
 \[
\begin{split}
\dom A&= \{x \in \sH :\, \exists\, y \in \sH \mbox{ with } (x,y) \in A\},\,\, \rm{domain}, \\
\ker A&= \{ x \in \sH:\, (x,0) \in A\}, \,\,\rm{kernel},  \\
\ran A&= \{y \in \sH :\, \exists\, x \in \sH \mbox{ with } (x,y) \in A\}, \,\,\rm{range}, \\
\mul A &= \{ y \in \sH:\, (0,y) \in A\}, \,\,\rm{multivalued\,\, part}.
\end{split}
\]
Note that the inverse   of $A$ is a linear relation given by $A^{-1}=\{ (y,x) :\, (x,y) \in A\}$. Hence there are the formal
identities $\dom A^{-1}=\ran A$ and $\ker A= \mul A^{-1}$.
In addition, recall the following definitions of the product and sum of linear relations $A$ and $B$;
and, in particular, of $A-\lambda$ and $\lambda A$ when $\lambda \in \dC$:
 \[
\begin{split}
AB&=\{ (x,z) \in \sH \times \sH:\, \exists\, z \in \sH \mbox{ with } (x,z) \in B, (z,y) \in A\}, \,\, \rm{product},\\
\lambda A &= \{ (x,\lambda y) \in \sH\times \sH :\, (x,y) \in A \},   \\
A+B&= \{(x,y+z) \in \sH\times \sH :\, \exists\, (x,y) \in A \mbox{ with } (x,z) \in B\}, \,\, \rm{sum}, \\
A-\lambda&=A-\lambda I=\{(x,y-\lambda x) \in \sH\times \sH :\, (x,y) \in A\},
\end{split}
\]
where $I=\{(x,x) \in \sH \times \sH:\,  x \in \sH\}$ stands for the identity operator.

\subsection{Root spaces and Jordan chains}

The usual \emph{point spectrum} $\sigma_p(A)$
is the set of all eigenvalues $\la\in\dC\cup\{\infty\}$ of the relation $A$:
\begin{equation}\label{ook333-}
\begin{aligned}
 \sigma_p(A) =\big\{\ \lambda \in \dC \cup \{\infty\}\ &:\
 \ker (A-\lambda) \neq \{0\}, \mbox{ if } \,\lambda \in \dC,  \\
 &\quad\ \ \mbox{ or }  \mul A \neq \{0\}, \mbox{ if } \, \lambda=\infty \big\}
\end{aligned}
\end{equation}
 The  \emph{root spaces} $\mathfrak{R}_{\lambda}(A)$ of $A$ for
$\lambda \in \dC\cup \{\infty\}$  are linear subspaces of $\sH$ defined by
\begin{equation}\label{rootsp}
\begin{split}
 &\mathfrak{R}_{\lambda}(A)=
 \spn\{  \ker(A-\lambda)^{i} :\; \lambda \in \dC, \ i\in \mathbb N\}, \\
&\mathfrak{R}_{\infty}(A)=
\spn \{ \mul A^{i}:\; i\in \mathbb N\}.
\end{split}
\end{equation}
Note that $x \in \sR_\lambda(A)$, $\lambda \in \dC$, if and only if
 for some $n \in \dN$ there exists a chain of elements in $\sH \times \sH$ of the form
\begin{equation}\label{jchain}
 (x_n,x_{n-1}+\lambda x_n ),
 (x_{n-1},x_{n-2}+\lambda x_{n-1} ),
 \dots,
 (x_{2}, x_{1}+\lambda x_{2}),
 (x_{1},\lambda x_{1} ) \in A
\end{equation}
such that $x=x_n$, the ``endpoint'' of \eqref{jchain};
for all $1 \leq i \leq n$ one has $(x_i,0) \in (A-\lambda)^i$.
If $x_1 \neq 0$, then the chain in \eqref{jchain}
is said to be a \textit{Jordan chain} for $A$ corresponding to the eigenvalue $\la \in \dC$. 
 Likewise, $y \in \sR_\infty(A)$ if and only if
 for some $m \in \dN$ there exists a chain of elements in $\sH \times \sH$ of the form
\begin{equation}\label{Jane}
   (0,y_1), \ (y_1,y_2), \ \dots \dots \dots , (y_{m-2}, y_{m-1}), \ (y_{m-1},y_m) \in A
 \end{equation}
such that $y=y_m$, the ``endpoint'' of \eqref{Jane}. If $y_1 \neq 0$, then the chain in \eqref{Jane}
is said to be a \textit{Jordan chain} for $A$ corresponding to the eigenvalue $\infty$;
for all $1 \leq i \leq m$ one has  $(0,y_i) \in A^i$.
 The \emph{total root space} $\sR_r(A)$ of $A$ is a linear subspace of $\sH$ defined by
\begin{equation}\label{totalrootsp}
 \sR_r(A)=\spn \setdef{\sR_\la(A)}{\lambda \in \dC\cup \{\infty\}};
\end{equation}
see \eqref{rootsp}.
Clearly,  an element belongs to $\sR_r(A)$ if and only if
it is the ``endpoint'' in the above sense
of a chain in \eqref{jchain} or of a chain in \eqref{Jane}.

\subsection{Singular chains}
The \emph{singular chain subspace} $\mathfrak{R}_{c}(A)$
of $A$  is a linear subspace of the total root space $\sR_r(A)$ defined by
\begin{equation}\label{Alejandro}
 \mathfrak{R}_{c}(A)=\mathfrak{R}_{0}(A)\cap \mathfrak{R}_{\infty}(A);
\end{equation}
cf. \cite{SandDeSn05}.
Note that $u \in \sR_c(A)$ if and only if for some $k \in \dN$ there is a chain of elements
of the form
\begin{equation}\label{schain}
  (0,u_k), \ (u_{k} , u_{k-1}), \ \dots \dots \dots , (u_{2}, u_{1}), \ (u_{1},0) \in A
\end{equation}
such that $u=u_l$ for some $1 \leq l \leq k$.
The chain in \eqref{schain} is said to be a \textit{singular chain} for $A$.
It is clear from \eqref{schain} that  $\sR_c(A) \subset \dom A \cap \ran A$, and that
$\sR_c(A) \neq \{0\}$ implies that $\ker A \cap \sR_c(A)$ and  $\mul A \cap \sR_c(A)$
are non-trivial.
 The  singular chain space $\sR_c(A)$ can also be written as follows (for a proof, see \cite{BSTW}):
 for any $\la, \mu\in\dC\cup\{\infty\}$ with $\la \ne \mu$
one has
\begin{equation}\label{eq:Rc-lambda-mu}
\sR_{c}(A)=\sR_{\la}(A)\cap\sR_{\mu}(A)
\end{equation}
so that, in particular,
\begin{equation}\label{il}
\sR_c(A) \subset \sR_\lambda(A), \quad \lambda \in \dC\cup \{\infty\}.
\end{equation}

\subsection{Proper point spectrum}

In order to discuss a reducing sum decomposition (see  Definition~\ref{decomm} below) in
terms of $\sR_c(A)$ and $\sR_r(A)$,
 one needs to consider a certain restriction of the point spectrum of~$A$.
 It is clear that if $\sR_c(A) \neq \{0\}$, then
$\sR_\lambda(A) \neq \{0\}$ for all $\lambda \in \dC \cup \{\infty\}$,
so that $\sigma_p(A) =\dC\cup\{\infty\}$. In fact,  it is known that, due to finite-dimensionality,
\begin{equation*}\label{AmazonPrime}
\sigma_p(A) =\dC\cup\{\infty\}\quad
\mbox{if and only if} \quad \sR_{c}(A)\neq\{0\},
\end{equation*}
see~\cite[Prop.~3.2, Thm.~4.4]{SandDeSn05}.
 The \textit{proper point spectrum}, see \cite{BSTW},
is a subset of the point spectrum $\sigma_{p}(A)$ and defined by
\begin{equation}\label{ook333}
  \sigma_{\pi}(A) = \setdef{\lambda\in\sigma_p(A)}{\sR_{\lambda}(A) \setminus \sR_c(A)  \neq\emptyset},
\end{equation}
 cf.\ \eqref{il}. The elements in $\sigma_\pi(A)$ are called the
\textit{proper eigenvalues} of $A$.
As a consequence of \eqref{ook333}, observe that
\begin{equation}\label{totalrootsp+}
 \sR_r(A)=\spn \setdef{\sR_\la(A)}{\lambda \in \sigma_\pi(A)}\cup\sR_{c}(A).
\end{equation}
Note that if $\sR_c(A)=\{0\}$, then $\sigma_{\pi}(A)=\sigma_{p}(A)$.
Entries of chains belonging to different proper eigenvalues in $\sigma_{\pi}(A)$
are linearly independent and hence
\begin{equation*}
    |\sigma_{\pi}(A) |\le \dim \sH,
\end{equation*}
see \cite{BSTW},
so that $\sigma_{\pi}(A)$ is a finite set,
since $\sH$ is assumed to be finite-dimensional.
The proper point spectrum $\sigma_\pi(A)$ of a linear relation $A$ will be the substitute
for the usual point spectrum $\sigma_p(A)$ in the operator case.

\subsection{Shift chains}

To complete the description of the structure of a linear relation $A$,
one needs to go beyond the total root space $\sR_r(A)$.
A collection of linearly independent elements $x_1,\ldots x_n$ in $\sH$
is called a \textit{shift chain}  if
\begin{equation}\label{multistage}
  (x_{1},x_{2}), \dots ,(x_{n-1},x_{n})\in A.
\end{equation}
Shift chains, in a sense, extend the notions of singular and Jordan chains:
\begin{itemize}
  \item if additionally $(x_n,0)\in A$, then~\eqref{multistage} is a Jordan chain at $0$,
  \item if additionally $(0,x_1)\in A$, then~\eqref{multistage} is a Jordan chain at $\infty$,
  \item if additionally $(0,x_1),(x_n,0)\in A$, then~\eqref{multistage} is a singular chain.
\end{itemize}
On the other hand, they exhibit completely different spectral properties: the linear relation spanned by the
elements in~\eqref{multistage} is an operator without point spectrum in $\dC$.
 A linear relation $A$ in a finite-dimensional linear space $\sH$
is said to be a \textit{multishift}
if $A$ has no eigenvalues in $\dC\cup\{\infty\}$ (i.e., if $A$ is a linear operator
without eigenvalues in $\dC$).
It will be shown that there exists a linear subspace $\sR_m(A) \subset \sH$,
spanned by entries of shift chains,
such that it complements the subspace $\sR_r(A)$, and
the graph restriction of $A$ to $\sR_{m}(A)$ is given by
\begin{equation}\label{aamm}
A_M=A\cap(\sR_m(A) \times \sR_m(A));
\end{equation}
cf. Theorem \ref{Thm:DecompAR-AM}. The relation $A_M$ in \eqref{aamm}
is a multishift.

\subsection{Reducing sum decompositions}

Here is a brief review of reducing sum decompositions for linear relations in a
linear space $\sH$. Recall that subspaces $\sH_j\subset \sH$ for $j=1,\ldots,n$
of~$\sH$ are said to form a \textit{direct sum}, denoted by
\begin{equation}
\label{red1-}
\sH_1 \oplus \sH_2 \oplus \cdots \oplus \sH_n
\end{equation}
if $0=x_1+x_2+\ldots +x_n$ with elements $x_j\in \sH_j$
implies that $x_j=0$ for $j=1,\ldots,n$. In particular, each $x\in \sH_1 + \sH_2 +\cdots + \sH_n$ admits
a sum $x=x_1+x_2+\ldots +x_n$ with unique elements $x_j\in \sH_j$ for $j=1,\ldots,n$.

A linear relation $A$ in a linear space $\sH$ is a linear
subspace of $\sH \times \sH$.
The componentwise sum $A_1 \hplus A_2$ of
linear relations $A_1$ and $A_2$ in a linear space  $\sH$ is defined as the sum of the subspaces in
$\sH \times \sH$:
\[
 A_1 \hplus A_2= \{ (x+u,y+v) \in \sH\times \sH:\, (x,y) \in A_1, \, (u,v)\in A_2\}.
\]
If a sum $A_1 \hplus A_2\hplus\cdots\hplus A_n$ of linear relations $A_j$ in $\sH$ is direct, it is
denoted by
\[
A_1 \oplus A_2 \oplus \cdots \oplus A_n.
\]
For any linear subspace $\sX \subset \sH$
one defines the \textit{graph restriction}
of $A$ to $\sX$ as $A'=A \cap (\sX \times \sX)$,
so that $A'$ is a linear relation in $\sX$.

Define the linear space $\sH(A)$ by $\sH(A)=\dom A+\ran A$.  Then clearly,
\begin{equation}\label{domran}
A \subset \sH(A) \times \sH(A) \subset \sH \times \sH.
\end{equation}
Hence $A$ coincides with its graph restriction to $\sH(A)$.
Moreover, one sees that $\sH(A)$ is the smallest subspace
$\sX \subset \sH$ with the
property that $A \subset \sX \times \sX \subset \sH \times \sH$.

\medskip

\begin{definition}\label{decomm}
Let $A$ be a linear relation in a linear space $\sH$. If
\begin{equation}\label{red1}
 A=A_1 \oplus A_2 \oplus \cdots \oplus A_n
\end{equation}
is  a direct sum of its graph restrictions  $A_j$ to subspaces $\sH_j\subset\sH$ for $j=1,\ldots,n$,
which form a direct sum $\sH_1 \oplus \sH_2 \oplus \cdots \oplus \sH_n$,
then the decomposition \eqref{red1} is called a reducing sum decomposition of~$A$ with
respect to $(\sH_1,\ldots,\sH_n)$.
\end{definition}
Instead of calling \eqref{red1} a reducing sum decomposition of~$A$ with
respect to $(\sH_1,\ldots,\sH_n)$,
it is frequently called a reducing sum decomposition of~$A$ with
respect to~\eqref{red1-}.

\begin{lemma}\label{redul}
Let $A$ be a linear relation in $\sH$ with a reducing sum decomposition~\eqref{red1} with respect to $(\sH_1,\ldots,\sH_n)$.
If $\sH=\dom A +\ran A$, then
\begin{equation}\label{red1++}
\sH_j = \dom A_j  + \ran A_j,
\quad j=1,\ldots, n.
\end{equation}
\end{lemma}

\begin{proof}
It follows from  \eqref{red1-}
and  \eqref{red1} that
\[
 \dom A=\dom A_1  \oplus \dots \oplus \dom A_n,
 \quad  \ran A=\ran A_1  \oplus \dots \oplus \ran A_n,
\]
so that since $\dom A_j+\ran A_j \subset \sH_j$ for  $j=1,\ldots, n$, one has
\[
  \dom A +\ran A =(\dom A_1+ \ran A_1)  \oplus \dots \oplus (\dom A_n +\ran A_n)
  \subset \sH_1 \oplus \dots \oplus \sH_n \subset \sH.
\]
Therefore,  the identity $\sH=\dom A +\ran A$ implies that  \eqref{red1++} holds.
\end{proof}

Note that Lemma~\ref{redul} implies that for any reducing sum decomposition~\eqref{red1} with respect to $(\sH_1,\ldots,\sH_n)$ one has that, since $A\subset \sH(A)\times\sH(A)$, $\sH_1 \oplus \dots \oplus \sH_n=\sH(A)$.

\medskip

Finally, it is emphasized that any reducing sum decomposition is intrinsically unique when the decomposition
 $\sH(A)=\sH_1 \oplus \dots \oplus \sH_n$ is fixed, as the linear relations $A_j$, $1\le j\le n$, are defined as the graph restrictions of $A$ to $\sH_j$.

\section{The completely singular part of a linear relation}
\label{Sec:SingJordanPart}

 Let $A$ be a linear relation in a finite-dimensional linear space $\sH$. The \textit{completely singular part} $A_S$ of $A$ is a linear relation defined
as the graph restriction of $A$  to the singular chain subspace $\sR_{c}(A)$:
\begin{equation}\label{srel}
 A_S =A \cap (\sR_c(A)\times \sR_c(A)).
\end{equation}
A linear relation $A$ in $\sH$ is called \emph{completely singular} if $A=A_S$. In this section it will be shown that $A_S$ is spanned by singular chains as in \eqref{schain}. The existence of such a basis was established in~\cite[Thm.~7.2]{SandDeSn05}.
The new proof in Theorem~\ref{Blaubeere} below is more in line with similar constructions in later sections
and it reveals some additional properties of the basis elements and the connection to the Weyr characteristic.

\medskip

{ The construction of the basis of singular chains is based on an appropriate  choice
of quotient spaces involving $\sR_c(A)$.
First, recall that $\ker A^{k}\subset \ker A^{k+1}$ for all $k \geq 1$.
The sequence of quotient spaces $\sK_k(A)$ is
defined by
\begin{equation}\label{tunesandI}
\sK_1(A)= \ker A \cap \sR_c(A), \quad
\sK_k(A):=\frac{\ker A^{k}\cap \sR_c(A)}{\ker A^{k-1}\cap\sR_c(A)},\ \ k\geq 2.
\end{equation}
Indeed, since the denominator is included in the numerator,
each quotient space $\sK_k(A)$, $k \geq 2$, is well defined.
The \textit{Weyr characteristic} of $A$ with respect to the
sequence of quotient spaces in~\eqref{tunesandI}
is defined as the sequence $(B_{k})_{k \geq 1}$ with
\begin{equation}\label{tunesandI+}
 B_k := \dim \sK_k(A), \quad k \geq 1.
\end{equation}
{ Observe that  if 
$\sR_{c}(A)=\{0\}$, then
$B_{k}=0$ for all $k \geq 1$. In this case $A_S$ is trivial.

\medskip

Now the case $\sR_{c}(A)\neq \{0\}$ will be considered.
Then the sequence in~\eqref{tunesandI+} is not
trivial, although ultimately the entries are zero.  To see this, observe that since the linear space $\sH$ is finite-dimensional the
number
\begin{equation}\label{d}
d=\min\setdef{k\in\dN}{\ker A^{k+1}\cap \sR_c(A)=\ker A^{k}\cap \sR_c(A)}
\end{equation}
is well defined.}

\begin{lemma}\label{Tja}
Let $A$ be a linear relation in a finite-dimensional space $\sH$ and
assume $\sR_{c}(A) \neq \{0\}$.  Let $d \geq 1$ be given by \eqref{d}, then for $k>d$
one has
$$
\ker A^{k}\cap \sR_c(A)= \ker A^{d}\cap \sR_c(A)\quad \mbox{and, hence,}\quad
B_k=0.
$$
Moreover, $B_1 \ge 1$.
\end{lemma}

\begin{proof}
Due to \eqref{d} it suffices to show that
$\ker A^{k}\cap \sR_c(A)=\ker A^{k+1}\cap \sR_c(A)$  for some $k \in \dN$
implies that
$\ker A^{k+1}\cap \sR_c(A)=\ker A^{k+2}\cap \sR_c(A)$.
Therefore, let $x\in \ker A^{k+2}\cap \sR_c(A)$.
Then there exist $x_1,\ldots, x_{k+1}$ such that
\begin{equation} \label{Desafortunadamente}
(x,x_{k+1}), (x_{k+1},x_{k})(x_k,x_{k-1}), \ldots, (x_1,0)\in A.
\end{equation}
As $x\in \sR_c(A)$, also $x_j\in \sR_c(A)$ for $1\leq k\leq k+1$. Moreover,
$x_{k+1}\in \ker A^{k+1}$ and, by assumption, $x_{k+1}\in \ker A^{k}\cap \sR_c(A)$.
Thus there exist $x'_1,\ldots, x'_{k-1}$ such that
$$
(x_{k+1},x'_{k-1}), (x'_{k-1},x'_{k-2}), \ldots, (x'_1,0)\in A.
$$
In combination with \eqref{Desafortunadamente} one obtains
$$
(x,x_{k+1}), (x_{k+1},x'_{k-1}), (x'_{k-1},x'_{k-2}), \ldots, (x'_1,0)\in A
$$
and $x\in \ker A^{k+1}\cap \sR_c(A)$ follows. This shows
\[
\ker A^{k+2}\cap \sR_c(A)\subset \ker A^{k+1}\cap \sR_c(A).
\]
The opposite inclusion follows from the fact that $\ker A^{k+1} \subset \ker A^{k+2}$.

To see that $B_1\ge 1$, observe that $\sR_c(A)\neq \{0\}$ implies $\ker A \cap \sR_c(A) \neq \{0\}$.
\end{proof}

\begin{theorem}\label{Blaubeere}
{ Let $A$ be a linear relation in a finite-dimensional space $\sH$ and
assume $\sR_{c}(A) \neq \{0\}$.  Let $d \geq 1$ be given by \eqref{d}, then
the Weyr characteristic $(B_k)_{k\geq 1}$ in \eqref{tunesandI+} satisfies
\begin{equation}\label{braun}
 B_{1} \geq B_{2} \geq \cdots \geq B_{d} \geq 1 \quad \mbox{and} \quad B_{k}=0,
 \quad k > d.
\end{equation}
Moreover,  there exist singular chains for $A$ of the following form}
\begin{equation} \label{Henk1}		
 \begin{array}{rl}
(0, x_{d}^i),\,
 (x_d^i, x_{d-1}^i),\, (x_{d-1}^i, x_{d-2}^i),\,  \dots  , (x_2^i, x_1^i),\ (x_1^i, 0), & \quad\quad \ 1 \leq i \leq B_d, \\
 (0,x_{d-1}^i),\; (x_{d-1}^i, x_{d-2}^i),\,  \dots ,  (x_2^i, x_1^i),\, (x_1^i, 0),         &B_d+1 \leq i \leq B_{d-1}, \\
 \ddots\qquad\qquad\vdots \quad\qquad \vdots\quad\ \ & \qquad\qquad\ \vdots \\
 (0,x_2^i),\, (x_2^i, x_1^i),\, (x_1^i, 0), & B_3+1 \leq i  \leq B_2, \\
 (0,x_1^i),\, (x_1^i, 0), & B_2+1 \leq i \leq B_1,
\end{array}
\end{equation}
where
$\{[x_{k}^{1}],\ldots,[x_{k}^{B_{k}}]\}$ is a basis  of $\sK_{k}(A)$, $1\le k\le d$,
and, consequently,
the elements in the set $\setdef{x_k^i}{1\le i\le B_k,\ 1\le k\le d}$ are linearly independent
in $\sH$. Furthermore, the completely singular part~$A_S$ defined in~\eqref{srel} admits the representation
\begin{equation}\label{TrapicheOakCasked}
\begin{split}
A_{S}&=\spn
\big\{(0, x_{k}^i),\,
 (x_k^i, x_{k-1}^i),\,  \dots  , (x_2^i, x_1^i),\ (x_1^i, 0):\, \\
&\hspace{5cm} B_{k+1}+1 \leq i \leq B_{k}, \ 1 \leq k \leq d \big\}.
\end{split}
\end{equation}
In particular, $\dom A_{S}=\ran A_{S} =\sR_{c}(A)$
and the total dimension of $A_{S}$ is
\begin{equation}\label{eq:dim-AS}
\dim A_{S} =2B_{1}+B_{2}+B_{3}+ \ldots+B_{d-1}+B_{d}.
\end{equation}
\end{theorem}

\begin{proof}
The main tool in the proof is the existence of linear relations
$A_{k}\subset \sK_{k}(A)\times \sK_{k-1}(A)$, $2 \leq k \leq d$, which are injective, i.e., $\ker A_k = \{0\}$.
From this, it follows that the sequence $(B_{k})_{k\ge 1}$ is nonincreasing.
The singular chains for $A$ will  be constructed via suitably chosen bases
in each of the quotient spaces $\sK_k(A)$, $1 \leq k \leq d$, beginning with $\sK_{d}(A)$
and working backwards to $\sK_{1}(A)$. This procedure is carried out
in a number of steps.

\medskip\noindent
\emph{Step 1}:
Let $1 \leq k \leq d$ and let $x \in [x]\in\sK_{k}(A)$.
Then there exists $y \in \sH$ such that
\[
 (x,y) \in A, \quad x\in\ker A^{k}\cap \sR_{c}(A), \quad y\in \ker A^{k-1} \cap \sR_{c}(A).
\]
For $k=1$ this means that $(x,0) \in A$.
To see the implication, note that by definition  $x\in\ker A^{k}\cap \sR_{c}(A)$.
Since $x\in\ker A^{k}$,  there is some $y\in\ker A^{k-1}$ such that $(x,y)\in A$ and $(y,0) \in A^{k-1}$.
Since $x \in \sR_c(A)$, it follows that $y \in \sR_c(A)$.
Therefore one concludes that $y\in \ker A^{k-1} \cap \sR_{c}(A)$.

\medskip\noindent
\emph{Step 2}:
Define the linear relation $A_{k}\subset \sK_{k}(A)\times \sK_{k-1}(A)$,  $2\le k \le d$, as follows:
\begin{equation}\label{ak}
A_k:= \setdef{([x],[y]) \in \sK_k(A)\times \sK_{k-1}(A)}{\exists\, (x',y')\in A \mbox{ with }
[x']=[x]\mbox{ and } [y']=[y]}.
\end{equation}
 By Step 1 it is clear that $A_k$ is defined on all of $\sK_{k}(A)$, $1\le k \le d$.

 Moreover, $A_k$ is injective for $k\ge 2$, that is, $\ker A_k =\{0\}$.
To see this, let $([x],[0])\in A_k$.
Then there exists $(x',y')\in A$ with
$[x]=[x'] \in\sK_{k}(A)$ and $[y']=[0]\in\sK_{k-1}(A)$.
Hence $x' \in \ker A^k\cap \sR_{c}(A)$ and $y' \in \ker A^{k-2}\cap \sR_{c}(A)$.
As $y' \in \sR_{c}(A)$ there exists   $z'\in \sR_{c}(A)$ with
$(z',y') \in A$.
Thus, $z'\in \ker A^{k-1}$ and $(x'-z',0)\in A$.
Since $\ker A\subset \ker A^{k-1}$, it follows that
 $$
x'=x'-z'+z' \in \ker A +\ker A^{k-1} \subset \ker A^{k-1}.
$$
This gives $[x]=[x']=0$ and shows that $A_{k}$ is injective. As a consequence,
the sequence $(B_{k})_{k\ge 1}$ is nonincreasing.
Recall that $B_k=0$  for $k>d$ by Lemma \ref{Tja}.

\medskip\noindent
\emph{Step 3}:
The construction of the singular chains for $A$ is associated with
the quotient spaces $\sK_d(A)$, \dots, $\sK_{1}(A)$, where $d \geq 1$.
Since $\dim \sK_d=B_d$, let
\begin{equation}\label{bd}
\{[v_d^1],\ldots,[v_d^{B_d}]\}
\end{equation}
be some basis for $\sK_d(A)$.

\medskip
First assume that $d=1$. In this case $(v_{1}^{i},0) \in A$, $1 \leq i \leq B_{1}$.
As $v_{1}^{i} \in \sR_{c}(A)$, there exist $z_{2}^{i} \in \sR_{c}(A)$ with
$(z_{2}^{i},v_{1}^{i}) \in A$, so that $z_{2}^{i}  \in \ker A^{2}\cap \sR_{c}(A)
=\ker A\cap \sR_{c}(A)$ by \eqref{d}. Hence $(z_{2}^{i},0) \in A$, and it follows that
$(0,v_{1}^{i}) = (z_{2}^{i},v_{1}^{i}) - (z_{2}^{i},0) \in A$. Thus with the choice $x_{1}^{i}=v_{1}^{i}$, $1 \leq i \leq B_{1}$,
the theorem has been proved when $d=1$.

\medskip
Next assume that $d \geq 2$.
 Then there are $d-1$ linear relations
 \[
A_d\subset \sK_{d}(A) \times \sK_{d-1}(A), \quad A_{d-1}
\subset \sK_{d-1}(A) \times \sK_{d-2}(A), \quad \dots,
\quad A_2\subset \sK_{2}(A) \times \sK_{1}(A),
\]
of the form \eqref{ak}.
With the choice \eqref{bd} it follows from Step 1
that there are elements $v_{d-1}^i$ such that for $1\le i \le B_d$:
\begin{equation*}
(v_d^i,v_{d-1}^i)\in A,
\quad  v_d^i\in\ker A^d \cap \sR_{c}(A),\quad
v_{d-1}^i\in \ker A^{d-1} \cap \sR_{c}(A).
\end{equation*}
As $v_d^i\in \sR_{c}(A)$, there exists $z_{d+1}^i\in \sR_{c}(A)$ such that
$(z_{d+1}^i,v_{d}^i)\in A$, and since $v_d^i\in \ker A^d$
this implies $z_{d+1}^i\in \ker A^{d+1}$.
Therefore, by \eqref{d}
one concludes that $z_{d+1}^i\in \ker A^d\cap \sR_{c}(A)$. Thus there exist numbers $\alpha_d^{i,j}$, $j=1,\ldots,B_d$, with
$$
z_{d+1}^i =y_{d-1}^{i}+\sum_{j=1}^{B_d}\alpha_d^{i,j} v_d^j,
$$
where $y_{d-1}^{i} \in \ker A^{d-1} \cap \sR_c(A)$.
Hence, one finds $u_{d-1}^{i}\in \ker A^{d-2}$ with
$$(y_{d-1}^{i},u_{d-1}^{i})\in A,$$
and
$$
\left(z_{d+1}^i, u_{d-1}^{i}+\sum_{j=1}^{B_d}\alpha_d^{i,j} v_{d-1}^j \right)
=\left(y_{d-1}^{i}
+\sum_{j=1}^{B_d}\alpha_d^{i,j} v_d^j , u_{d-1}^{i}
+\sum_{j=1}^{B_d}\alpha_d^{i,j} v_{d-1}^j \right)
\in A,
$$
which, via $(z_{d+1}^i,v_{d}^i)\in A$, implies
$$
\left(0, v_{d}^i- u_{d-1}^{i}- \sum_{j=1}^{B_d}\alpha_d^{i,j} v_{d-1}^j\right)\in A.
$$
This result suggests to define the elements $x_d^i$, $1 \leq i \leq B_{d}$, by
$$
x_{d}^i:= v_{d}^i- u_{d-1}^{i}-\sum_{j=1}^{B_d}\alpha_d^{i,j} v_{d-1}^j,
 \quad 1 \leq i \leq B_d.
$$
Clearly, $x_d^i \in \ker A^{d} \cap \sR_{c}(A)$ for $1 \leq i \leq B_{d}$,
and they provide a basis for the quotient space $\sK_d(A)$:
\begin{equation}\label{A}
\spn \{[x_d^1],\ldots,[x_d^{B_d}]\} = \sK_d(A) \quad \mbox{with}
\quad (0, x_{d}^i) \in A, \quad 1 \leq i \leq B_d.
\end{equation}
Again, by Step 1, with the elements $x_d^i$ from \eqref{A} there exist  elements $x_{d-1}^i$ such that
for $1\le i \le B_d$:
\begin{equation*}\label{gris}
 (x_d^i,x_{d-1}^i)\in A, \quad
x_d^i \in \ker A^d \cap \sR_c(A), \quad
  x_{d-1}^i\in \ker A^{d-1} \cap \sR_{c}(A).
 \end{equation*}
 Observe that by definition
\[
 ([x_d^i],[x_{d-1}^i])\in A_d,
\]
so that by Step~2 the elements $[x_{d-1}^i]$,
$1\le i \le B_d$, are linearly independent in $\sK_{d-1}(A)$.
Since $\dim \sK_{d-1}(A)=B_{d-1} \geq B_d$,
one can enlarge the family $\{[x_{d-1}^1],\ldots,[x_{d-1}^{B_d}]\}$
by choosing elements $[v_{d-1}^{B_d+1}],\ldots,[v_{d-1}^{B_{d-1}}]$ such that
$$
\spn\{[x_{d-1}^1],\ldots,[x_{d-1}^{B_d}],[v_{d-1}^{B_d+1}],\ldots,[v_{d-1}^{B_{d-1}}] \} = \sK_{d-1}(A).
$$
 By Step 1 there exist elements $v_{d-2}^i$ such that for $B_d+1\le i \le B_{d-1}$:
\begin{equation*}
(v_{d-1}^i,v_{d-2}^i)\in A, \quad
v_{d-1}^i\in\ker A^{d-1} \cap \sR_{c}(A),\quad
v_{d-2}^i\in \ker A^{d-2} \cap \sR_{c}(A).
\end{equation*}
Similar to the procedure leading to~$x_d^i$ it is then possible to find elements $x_{d-1}^i \in \ker A^{d-1} \cap \sR_{c}(A)$,
$B_d+1\le i \le B_{d-1}$, and to obtain a basis for the quotient space $\sK_{d-1}(A)$:
\begin{equation}\label{B}
\{[x_{d-1}^1],\ldots,[x_{d-1}^{B_{d-1}}]\} \quad \mbox{with} \quad
\begin{array}{l}
(x_d^i,x_{d-1}^i)\in A \quad \mbox{for} \quad 1\le i \le B_d, \\
(0, x_{d-1}^i) \in A \quad \mbox{for}  \quad B_d+1\le i \le B_{d-1},
\end{array}
\end{equation}
so that the basis $\{[x_{d-1}^1],\ldots,[x_{d-1}^{B_{d-1}}]\}$
of $\sK_{d-1}(A)$ is in the desired form.

\medskip
Continuing in this way by induction, one finds
as successors to \eqref{B},
 that for $1 \leq j \leq d-1$ there exist elements
\[
x_{d-j}^i \in \ker A^{d-j}\cap \sR_{c}(A), \quad 1 \leq i \leq B_{d-j},
\]
which give a basis for the quotient space $\sK_{d-j}(A)$:
\begin{equation}\label{C}
\{[x_{d-j}^1],\ldots,[x_{d-j}^{B_{d-j}}]\} \quad \mbox{with} \quad
\begin{array}{l}
(x_{d-j+1}^i,x_{d-j}^i)\in A \quad \mbox{for} \quad 1\le i \le B_{d-j+1}, \\
(0, x_{d-j}^i) \in A \quad \mbox{for}  \quad B_{d-j+1}+1\le i \le B_{d-j},
\end{array}
\end{equation}
so that the basis $\{[x_{d-j}^1],\ldots,[x_{d-j}^{B_{d-j}}]\}$
of $\sK_{d-j}(A)$ is in the desired form.
Note that for $j=d-1$ the construction implies that $(x_{1}^{i},0) \in A$, $1 \leq i \leq B_{1}$.
Hence the assertion concerning the existence of singular chains for the
completely singular part $A_{S}$ has been proved.

\medskip\noindent
\emph{Step 4}:
It follows from the construction in Step 3 that
$\{[x_{k}^{1}],\ldots,[x_{k}^{B_{k}}]\}$ is a basis  of $\sK_{k}(A)$, $1\le k\le d$.
 Then the representatives
\[
\setdef{x_k^i}{1\le i\le B_k,\ 1\le k\le d}
\]
are linearly independent in $\sH$. To see this, assume that for some $c_k^i\in\dC$
\[
\sum_{k=1}^{d}\sum_{i=1}^{B_k} c_k^i x_k^i=0.
\]
Forming equivalence classes in $\sK_d(A)$, it follows by definition
that $[x_k^i] = [0] \in\sK_d(A)$ for $i=1,\ldots, B_k$ and $k=1,\ldots, d-1$, and since $\{[x_d^1],\ldots,[x_d^{B_d}]\}$ is a basis of $\sK_d(A)$ the reduced equality
\[
\sum_{i=1}^{B_d} c_d^i [x_d^i]=0
\]
implies that $c_d^1 = \ldots c_d^{B_d} = 0$. Forming equivalence classes in $\sK_{d-1}$ and
proceeding in a similar way, it ultimately follows that
$c_k^i=0$ for all the coefficients, which proves the claim.

\medskip\noindent
\emph{Step 5}:  It will be shown that~\eqref{TrapicheOakCasked} holds.
In fact, by the construction in Step 3 it suffices to show that $A_{S}$ is contained in the right-hand side of
\eqref{TrapicheOakCasked}. To this end,
let $(x,y)\in A_S$, so that $(x,y) \in A$ and $x,y \in \sR_{c}(A)$.
If $y=0$, then
\[
    x \in \ker A \cap \sR_c(A) = \sK_1(A) = \spn\{x_1^1,\ldots,x_1^{B_1}\}
\]
and the assertion is shown. If $y\neq 0$, then $y\in \ker A^k$ for some $1\leq k\leq d$ and hence $x\in \ker A^{k+1}$.
Choose maximal $k\geq 1$ with the property
$[y]\in \sK_{k}(A)\setminus\{[0]\}$. If $k=d$, then
$$
y= y_{d-1}+\sum_{i=1}^{B_{d}}\gamma^i x_{d}^i,
$$
where $\gamma^i \in \dC$, $i=1,\ldots,B_d$, and  $y_{d-1} \in \ker A^{d-1}\cap \sR_{c}(A)$.
It follows that
$$
(x,y)=(x,y_{d-1})+\sum_{i=1}^{B_{d}}\gamma^i \left(0 , x_{d}^i\right),
$$
and it remains to show
that $(x,y_{d-1})$ is contained in the right-hand side of \eqref{TrapicheOakCasked}.
Note that $x\in \ker A^{d+1}\cap \sR_{c}(A)=\ker A^{d}\cap \sR_{c}(A)$ by~\eqref{d}.
To continue by an inductive argument, assume now that $k<d$.
Then one can write $y$ as
$$
y= y_{k-1}+\sum_{i=1}^{B_{k}}\gamma^i x_{k}^i,
$$
where $\gamma^i \in \dC$, $i=1,\ldots,B_k$, and  $y_{k-1} \in \ker A^{k-1}\cap \sR_{c}(A)$.
Hence, there exists $y_{k}\in \ker A^k\cap \sR_{c}(A)$ with $(y_k,y_{k-1}) \in A$ and it follows that
\[
\left( y_{k}+\sum_{i=1}^{B_{k+1}}\gamma^i x_{k+1}^i,\, y\right)
=( y_{k}, y_{k-1}) +\sum_{i=1}^{B_{k+1}}\gamma^i (x_{k+1}^i,  x_{k}^i)
+\sum_{i=B_{k+1}+1}^{B_{k}}\gamma^i (0,x_{k}^i)   \in A,
\]
where $(0, x_{k}^i)\in A$ for $B_{k+1}+1\leq i\leq B_{k}$ was used, see Step 3.
Since $[y_{k}] = [0] \in \sK_{k+1}(A)$, this gives
$$
\left( \left[\sum_{i=1}^{B_{k+1}}\gamma^i x_{k+1}^i\right],\, [y]\right) \in A_{k+1},
$$
whereas it is clear that  $([x],[y]) \in A_{k+1}$.
By Step 2, the linear relation $A_{k+1}$ is injective, so that
$$
[x]= \left[\sum_{i=1}^{B_{k+1}}\gamma^i x_{k+1}^i\right]
$$
in $\sK_{k+1}(A)$ and therefore $x= x_k+\sum_{i=1}^{B_{k+1}}\gamma^i x_{k+1}^i$ with some $x_k\in \ker A^k\cap \sR_c(A)$. Hence, one obtains
$$
(x,y) = (x_k,y_{k-1}) + \sum_{i=1}^{B_{k+1}}\gamma^i \left( x_{k+1}^i, x_{k}^i\right)  + \sum_{i=B_{k+1}+1}^{B_{k}}\gamma^i \left(0,
x_{k}^i\right).
$$
From this it follows that
 to show $(x,y) \in (\ker A^{k+1}\cap \sR_c(A)) \times (\ker A^k\cap \sR_c(A))$
 is a linear combination of elements from the
 right-hand side of ~\eqref{TrapicheOakCasked}, it is sufficient
to show the same for $(x_k,y_{k-1})\in (\ker A^{k}\cap \sR_c(A)) \times (\ker A^{k-1}\cap \sR_c(A))$.
By repeating the above procedure
one arrives at elements $(x_1,y_0)=(x_{1},0)$,
where $x_{1}\in \ker A \cap \sR_{c}(A)=  \sK_{1}(A)$.
Obviously, $x_{1}$ can be written as
a linear combination of the elements $x_i^1$, $i=1,\ldots,B_1$,
and therefore, $(x_{1},0)$ is a
linear combination of elements of the form
$(x_i^1,0)$. Thus $(x,y)$ belongs to the right-hand side of
\eqref{TrapicheOakCasked}.

\medskip\noindent
\emph{Step 6}: It remains to show~\eqref{eq:dim-AS}, which directly follows from
\eqref{Henk1}.
\end{proof}

\section{The root part of a linear relation}
\label{sec3}

Let $A$ be a linear relation in a finite-dimensional linear space $\sH$,
and let $A_S$ be its completely singular part~\eqref{srel}.
The linear relation $A_{\la}$, $\lambda \in \dC \cup \{\infty\}$, is defined
as the graph restriction of $A$ to the root space $\sR_{\la}(A)$:
 \begin{equation}\label{deux}
A_{\la}=A \cap (\sR_\la(A)\times \sR_\la(A)).
\end{equation}
By~\eqref{il} one has that $A_{S} \subset A_{\lambda}$,  $\lambda \in \dC \cup \{\infty\}$.
 The \emph{root part} $A_R$ of $A$ is a linear relation
defined as the graph restriction of $A$  to the total root subspace $\sR_{r}(A)$ in \eqref{totalrootsp}:
\begin{equation}\label{eq:AR}
A_R =A \cap (\sR_r(A) \times \sR_r(A)).
\end{equation}
Hence it is clear that $A_{S} \subset A_{\lambda} \subset  A_{R}$, $\lambda \in \dC \cup \{\infty\}$.

If $\lambda \not\in \sigma_{\pi}(A)$,
then $\sR_{c}(A)=\sR_\lambda(A)$ and $A_{S}=A_{\lambda}$.
However, if $\lambda \in \sigma_{\pi}(A)$, then
$\sR_{c}(A)\subset \sR_{\lambda}(A)$ and
$A_{S} \subset A_{\lambda}$ with strict inclusion.
there is a
reducing sum decomposition  for $A_{\lambda}$ of the form
\begin{equation}\label{jor}
A_{\lambda}=A_{S} \oplus J_{\lambda}(A),
\end{equation}
where $J_{\lambda}(A)$ is a Jordan operator (if $\lambda \in \dC$) or a Jordan relation (if $\lambda=\infty$), see Definition~\ref{NocheEnMedellin} below. 
Moreover, it will be shown that $J_\lambda(A)$ is spanned
by the corresponding Jordan chains as in \eqref{jchain} or \eqref{Jane};
see Theorems~\ref{Thm:existenceAla}
and~\ref{Thm:existenceAlaInfty} below.
Then in Theorem~\ref{Thm:DecompAS-AJ} below it will be shown that
 \begin{equation}\label{LanadelRey-}
A_R = A_S \oplus J_{\la_1}(A) \oplus \cdots\oplus J_{\la_l}(A)\oplus J_\infty(A),
\end{equation}
where $\lambda_{1}, \dots, \lambda_{l} \in \sigma_{\pi}(A)$ and
$\infty \in \sigma_{\pi}(A)$ are the proper eigenvalues of $A$.
If $\infty \notin \sigma_{\pi}(A)$, then the term
 $J_\infty(A)$ in \eqref{LanadelRey-} is absent.

In this section, the identity \eqref{jor}
will first be shown for the case $\lambda=0$ in Lemma \ref{nullketten}.
The case when $\lambda \in \dC$ will be obtained in Theorem \ref{Thm:existenceAla}
by a shift of the relation from Lemma \ref{nullketten}.
 Likewise, the case when $\lambda=\infty$ will be obtained
in Theorem \ref{Thm:existenceAlaInfty}
by an inversion of the relation from Lemma \ref{nullketten}.

\begin{definition}\label{NocheEnMedellin}
A linear relation $A$ in a finite-dimensional linear space $\sH$ is called a \emph{Jordan operator} in $\sH$ {\rm (}corresponding to $\lambda\in\dC${\rm )}, if $\dom A = \sH$, $\mul A = \{0\}$, and $\sigma_p(A) = \{\lambda\}$. The relation $A$ is called a \emph{Jordan relation} in $\sH$
{\rm (}corresponding to $\infty${\rm )}, if $A^{-1}$ is a Jordan operator {\rm (}corresponding to $0\in\dC${\rm )}.
\end{definition}
Note that Jordan operators are essentially matrices and Jordan relations correspond to
injective multi-valued operators. In particular, it is clear that a Jordan relation (corresponding to $\infty$) satisfies $\ran A = \sH$, $\ker A = \{0\}$, and $\sigma_p(A) = \{\infty\}$.

{ The construction of the Jordan operator $J_{\lambda}(A)$ for $\lambda=0$
is based on an appropriate  choice of a sequence of quotient spaces involving $A$.
The sequence of quotient spaces $\sV_k(A)$ is defined by
\begin{equation}\label{vvv00}
 \sV_1(A)=\frac{\ker A+\sR_c(A)}{\sR_c(A)}, \quad
 \sV_k(A)=\frac{\ker A^{k}+\sR_c(A)}{\ker A^{k-1}+\sR_c(A)},\quad k\geq 2.
\end{equation}
Indeed, since the denominator is contained in the numerator,
each quotient space $\sV_k(A)$, $k \geq 1$, is well defined.
Define the sequence $(d_{k})_{k \geq 1}$ by
\begin{equation}\label{vvv00+}
 d_k := \dim \sV_k(A),\quad k\geq 1.
\end{equation}
Note that the three conditions $0 \not\in \sigma_{\pi}(A)$,
$\sR_{c}(A)=\sR_{0}(A)$ and $d_{k}=0$ for all $k \geq 1$ are all equivalent.

\medskip

Now the case $0 \in \sigma_{\pi}(A)$ or, equivalently,
$\sR_{c}(A) \subsetneq \sR_{0}(A)$, will be considered.
Then the sequence in~\eqref{vvv00+} is not trivial, although ultimately the entries are zero. To see this,
observe that since the space $\sH$ is finite-dimensional, the number
\begin{equation}\label{neerst}
v=\min\setdef{k\in\dN}{\ker A^{k+1}+\sR_c(A)=\ker A^{k}+\sR_c(A)}
\end{equation}
is well defined.

\begin{lemma}
Let $A$ be a linear relation in a finite-dimensional space $\sH$ and
assume $0 \in \sigma_{\pi}(A)$.  Let $v \geq 1$ be given by \eqref{neerst}, then for $k>v$ one has
$$
\ker A^{k}+ \sR_c(A)= \ker A^{d}+ \sR_c(A)\quad \mbox{and, hence,}\quad
d_k=0.
$$
Moreover, $d_1\ge 1$.
\end{lemma}

\begin{proof}
In view of \eqref{neerst}, by induction it is sufficient to show that
$\ker A^{k}+ \sR_c(A)=\ker A^{k+1}+ \sR_c(A)$  for some natural number $k$
implies that
$\ker A^{k+1}+ \sR_c(A)=\ker A^{k+2}+ \sR_c(A)$. Let $x\in \ker A^{k+2}+ \sR_c(A)$ and
write $x=x_{k+2}+x_c$ with $x_{k+2}\in\ker A^{k+2}$ and $x_c\in \sR_c(A)$.
Then there exist $x_1,\ldots, x_{k+1}$ such that
\begin{equation*} \label{BistDuDabei?}
(x_{k+2},x_{k+1}), (x_{k+1},x_{k})(x_k,x_{k-1}), \ldots, (x_1,0)\in A.
\end{equation*}
As $x_{k+1}\in \ker A^{k+1}\subset \ker A^{k+1}+ \sR_c(A)=\ker A^{k}+ \sR_c(A)$ by assumption, there exists $x'_c\in \sR_c(A)$ such that $x_{k+1}-x'_c\in \ker A^{k}$ and
there exist $x'_1,\ldots, x'_{k-1}$ such that
$$
(x_{k+1}-x'_c,x'_{k-1}), (x'_{k-1},x'_{k-2}), \ldots, (x'_1,0)\in A.
$$
As $x'_c\in \sR_c(A)$ there is $y'_c\in \sR_c(A)$ with $(y'_c,x'_c) \in A$
and one concludes
$$
(x_{k+2}-y'_c,x_{k+1}-x'_c), (x_{k+1}-x'_c,x'_{k-1}), (x'_{k-1},x'_{k-2}), \ldots, (x'_1,0)\in A
$$
so that $x_{k+2}\in \ker A^{k+1}+ \sR_c(A)$ follows. This shows
$\ker A^{k+2}+ \sR_c(A)\subset \ker A^{k+1}+ \sR_c(A)$.
The opposite inclusion follows from the fact that $\ker A^{k+1} \subset \ker A^{k+2}$.

To see that $d_1\ge 1$, observe that $0 \in \sigma_{\pi}(A)$ implies $\sR_c(A)\subsetneq \ker A$. Hence, $\sR_c(A) \subsetneq \sR_c(A) + \ker A$, which proves the claim.
\end{proof}

\begin{lemma}\label{nullketten}
Let $A$ be a linear relation in a finite-dimensional space $\sH$ with
 $0 \in \sigma_{\pi}(A)$.
Let $v \geq 1$ be given by \eqref{vvv00}, then
the sequence $(d_k)_{k\geq 1}$ in \eqref{vvv00+} satisfies
\begin{equation*}\label{braunp}
 d_{1} \geq d_{2} \geq \cdots  \geq d_{v} \geq 1 \quad \mbox{and} \quad d_{k}=0,
 \quad k > v.
\end{equation*}
%
%
%
%
Moreover, there exist
Jordan chains for $A$ corresponding to the eigenvalue $0$
of the following form:
\begin{equation}\label{Henk2}
\begin{array}{rl}
 (x_v^i, x_{v-1}^i),\ (x_{v-1}^i, x_{v-2}^i),\, \dots  , (x_2^i, x_1^i),\ (x_1^i, 0), &  \quad\quad \; 1 \leq i \leq d_v, \\
 (x_{v-1}^i, x_{v-2}^i),\, \dots ,  (x_2^i, x_1^i),\ (x_1^i, 0),         &d_v+1 \leq i \leq d_{v-1}, \\
\ddots\qquad\vdots \quad\qquad \vdots\quad\ \ & \qquad\qquad \vdots \\
   (x_2^i, x_1^i),\ (x_1^i, 0), & d_3+1 \leq i  \leq d_2, \\
  (x_1^i, 0), & d_2+1 \leq i \leq d_1,
\end{array}
\end{equation}
where
$\{[x_{k}^{1}],\ldots,[x_{k}^{d_{k}}]\}$ is a basis  of $\sV_{k}(A)$ in \eqref{vvv00}, $1\le k\le v$,
and, consequently, the elements in the set
$\setdef{x_k^i}{1\le i\le d_k,\ 1\le k\le v}$ are linearly independent in $\sH$.
Then the linear space $\sR_{0}(A)$ has the direct sum decomposition
\begin{equation}\label{ilme}
\sR_0(A)= \sR_c(A) \oplus \sX_0(A),
\end{equation}
where the space $\sX_0(A)$ is given by
\[
\sX_{0}(A)=\spn \setdef{ x_k^i}{1\le i\le d_k,\ 1\le k\le v}.
\]
Furthermore, with respect to \eqref{ilme}, the graph restriction of $A$ to $\sR_0(A) \times \sR_0(A)$,
\[
A_{0}=A \cap (\sR_0(A) \times \sR_0(A))
\]
has the reducing sum decomposition
\begin{equation}\label{EternalLoveAngelicaZapata}
A_0 = A_S \oplus J_0(A),
\end{equation}
where the linear relation $J_{0}(A)=A\cap (\sX_0(A) \oplus \sX_0(A))$ admits the representation
\begin{equation}\label{Bailando}
\begin{split}
 J_{0}(A)&=\spn
\big\{
 (x_k^i, x_{k-1}^i),\,  \dots  \dots , (x_2^i, x_1^i),\ (x_1^i, 0):\, \\
&\hspace{5cm} d_{k+1}+1 \leq i \leq d_{k}, \ 1 \leq k \leq v\big\}.
\end{split}
\end{equation}
In fact, $J_0(A)$ is a Jordan operator in $\sX_0(A)$ corresponding to $0\in\dC$ and the total dimension of $J_0(A)$ is
\begin{equation}\label{eq:dim-J0}
\dim J_0(A) =d_{1}+d_{2}+ \ldots+d_{v-1}+d_{v}.
\end{equation}
\end{lemma}

\begin{proof}
 \medskip\noindent
The main tool in the proof is the existence of linear operators
$A_{k}:\sV_{k}(A)\to \sV_{k-1}(A)$, $2 \leq k \leq v$, which are injective.
From this, it follows that the sequence $(d_{k})_{k\ge 1}$ is nonincreasing.
The Jordan chains for $A$ will  be constructed via suitably chosen bases
in each of the quotient spaces $\sV_k(A)$, $1 \leq k \leq v$, beginning with $\sV_{v}(A)$
and working backwards to $\sV_{1}(A)$. This procedure is carried out
in a number of steps.

\medskip\noindent
\emph{Step 1}:
Let $1 \leq k \leq v$ and $[x]\in\sV_{k}(A)$.
Then  there exist $x_1, y_1 \in \sH$ such that
\[
  (x_1,y_1) \in A, \quad x_1 \in \ker A^{k}, \quad y_1 \in \ker A^{k-1}
\quad \mbox{with} \quad [x_1]=[x].
\]
To see this, observe that for $x \in [x]$ one has $x=x_1+x_2$ with $x_1\in\ker A^{k}$
and $x_2\in\sR_c(A)$. Hence,  there is some $y_1\in\ker A^{k-1}$
such that $(x_1,y_1)\in A$. It follows from
\[
x-x_1\in\sR_c(A)\subset \ker A^{k-1}+\sR_{c}(A),
\]
that $[x]=[x_1]$.

\medskip\noindent
\emph{Step 2}:
Define the linear relation $A_{k}\subset \sV_k(A)\times \sV_{k-1}(A)$, $2 \le k \le v$, by:
\begin{equation}\label{ak1}
\begin{split}
A_k& := \big\{ ([x],[y]) \in \sV_k(A) \times \sV_{k-1}(A)  :\, \\
&\hspace{3cm} \exists\, (x',y')\in A \mbox{ with } [x']=[x]\mbox{ and } [y']=[y] \big\}.
\end{split}
\end{equation}
Since $\ker A^{k-1} \subset \ker A^{k-1}+\sR_c(A)$, it follows from Step 1
that $A_{k}$ is defined on all of $\sV_{k}(A)$, $2 \leq k \leq v$.

Moreover, $A_k$ is an operator.
To prove this, let $([0],[y])\in A_k$.
Hence, there exists $(x',y')\in A$ with $y'\in\ker A^{k-1}+\sR_{c}(A)$, such that
$[x']=[0]\in\sV_{k}(A)$ and $[y']=[y]\in\sV_{k-1}(A)$. In particular,
\[
x'=x'_1+x'_2 \quad \mbox{with} \quad  x'_1\in\ker A^{k-1}
\quad \mbox{and}  \quad x'_2\in\sR_c(A).
\]
Therefore, there exist $y'_1\in\ker A^{k-2}$ with $(x'_1,y'_1)\in A$,
and  $y'_2\in \sR_c(A)$ with $(x'_2,y'_2)\in A$.
It follows that $(0,y'-y'_1-y'_2)\in A$, thus $y'_3:=y'-y'_1-y'_2\in \sR_{\infty}(A) \cap \sR_{0}(A) = \sR_c(A)$.
Note that
\[
y'=y'_1+y'_2+y'_3 \quad \mbox{with}  \quad y'_1\in\ker A^{k-2}
\quad \mbox{and} \quad y'_2+y'_3 \in\sR_c(A),
\]
which implies that $[y]=[y']=[0]\in \sV_{k-1}(A)$.

Furthermore, $A_k$ is injective.
To prove this, let $([x],[0])\in A_k$.
Hence, there exists   $(x',y')\in A$ with
$[x]=[x']\in \sV_{k}(A)$ and $[y']=[0]\in \sV_{k-1}(A)$.
In particular,
\[
y'=y'_1+y'_2 \quad \mbox{with} \quad  y'_1\in\ker A^{k-2}
\quad \mbox{and} \quad y'_2\in\sR_c(A).
\]
Therefore, there exists $x'_2\in\sR_c(A)$ with $(x'_2,y'_2)\in A$.
Thus it follows from
\[
(x'-x'_2,y'_1)\in A \quad \mbox{and}  \quad y'_1\in\ker A^{k-2},
\]
that $x'-x'_2\in\ker A^{k-1}$. In other words,  $x' \in \ker A^{k-1} +\sR_{c}(A)$,
which implies that $[x]=[x']=[0]\in\sV_{k}(A)$.

\medskip\noindent
\emph{Step 3}:
The construction of the Jordan chains for $A$,
corresponding to the eigenvalue $0 \in \sigma_\pi(A)$,
is associated with the quotient spaces
$\sV_v(A)$, \dots, $\sV_{1}(A)$, where $v \geq 1$.
 Since $\dim \sV_{v}(A)=d_{v}$, let
\begin{equation}\label{bd1}
\{[x_v^1],\ldots,[x_v^{d_v}]\}
\end{equation}
be a basis of $\sV_v(A)$.

\medskip
First assume that $v=1$. Then by Step 1 each $[x_1^i]$ has a representative
$x_1^i$ such that $(x_1^i,0) \in A$, $1 \leq i \leq d_1$. This agrees with the statement of
the lemma.

\medskip
Next assume that $v \geq 2$.
Then there are  $v-1$ mappings:
\[
A_v: \sV_{v}(A) \to \sV_{v-1}(A), \,A_{v-1}: \sV_{v-1}(A) \to \sV_{v-2}(A),
\,\dots, \, A_2: \sV_{2}(A) \to \sV_{1}(A),
\]
of the form \eqref{ak1}.
By Step~1, without loss of generality one may assume that the choice \eqref{bd1} is such that
there are elements $x_{v-1}^i$ so that for $1\le i \le d_v$:
\begin{equation}\label{een}
(x_v^i,x_{v-1}^i)\in A, \quad  x_v^i\in\ker A^v, \quad  x_{v-1}^i\in \ker A^{v-1}.
\end{equation}
As $x_{v-1}^{i}\in \ker A^{v-1}$,
$1\le i \le d_v$, there exists  $x_{v-2}^{i}\in \ker A^{v-2}$ with $(x_{v-1}^i, x_{v-2}^i)\in A$.
In addition,  by the definition of $A_v$, one has
\[
A_v[x_v^i]=[x_{v-1}^i], \quad 1\le i \le d_v,
\]
Since $A_{v}: \sV_v(A)\to \sV_{v-1}(A)$ is injective by Step 2,
the elements $[x_{v-1}^i]$, $1\le i \le d_v$,
are linearly independent in the space $\sV_{v-1}(A)$.
Now choose
$[x_{v-1}^{d_v+1}],\ldots,[x_{v-1}^{d_{v-1}}]\in\sV_{v-1}(A)$ such that
\begin{equation}\label{bd2}
 \{[x_{v-1}^1],\ldots,[x_{v-1}^{d_v}], [x_{v-1}^{d_v+1}],\ldots,[x_{v-1}^{d_{v-1}}]\}
\end{equation}
forms a basis of $\sV_{v-1}(A)$. Hence, by Step~1, without loss of generality one may assume that $x_{v-1}^{d_v+1},\ldots,x_{v-1}^{d_{v-1}}$ are chosen such that there exist elements $x_{v-2}^i$ so that for
$d_v+1\le i \le d_{v-1}$:
\[
(x_{v-1}^i,x_{v-2}^i)\in A,
\quad x_{v-1}^i\in\ker A^{v-1},\ x_{v-2}^i\in \ker A^{v-2}.
\]
Furthermore, via the basis in \eqref{bd2} for $\sV_{v-1}(A)$,  one finds elements
 $x_{v-2}^i \in \ker A^{v-2}$  such that $(x_{v-1}^i, x_{v-2}^i)\in A$ for $1\le i \le d_{v-1}$.

Repeating this procedure a number of times, one finally arrives at
a basis of $\sV_{1}(A)$ of the form $\{[x_{1}^{1}],\ldots,[x_{1}^{d_{1}}]\}$ with elements $x_1^i \in \ker A$ such that
\[
 (x_1^i, 0) \in A, \quad 1 \leq i \leq d_1.
\]

\medskip\noindent
\emph{Step 4}:
The elements in
$\sX_0(A) = \setdef{x_k^i}{1\le i\le d_k,\ 1\le k\le v}$
 are linearly independent.
In fact, this  follows in the same way as in Step 4 of the proof
of Theorem~\ref{Blaubeere}, when
one replaces $d$ by $v$, $\sK_d(A)$ by $\sV_v$, and $B_k$ by $d_k$.

\medskip\noindent
\emph{Step 5}: The direct sum decomposition in \eqref{ilme} holds. In order to see that $\sX_0(A) \cap \sR_c(A) = \{0\}$,
one uses a similar argument in Step~4:
if $y\in \sR_c(A)$, then $[y]=0$ in all spaces $\sV_{k}$, and so $y\in \sX_0(A)\cap \sR_c(A)$ implies $y=0$.

 It is clear that $\sR_c(A) \oplus \sX_0(A) \subset\sR_0(A)$. To see that equality holds,
note that the definition of $v$ gives
$
\sR_0(A)=\ker A^v+\sR_c(A),
$
while $\sR_c(A)=\ker A^0+\sR_c(A)$. Hence, it follows from  \eqref{vvv00} that
\[
 \dim \frac{\sR_0(A)}{\sR_c(A)} =\dim \frac{\ker A^v+\sR_c(A)}{\ker A^0+\sR_c(A)}
  =\sum_{k=1}^v \dim \sV_k  =\dim \sX_0(A),
\]
which completes the argument. Thus \eqref{ilme} has been shown.

\medskip\noindent
\emph{Step 6}:
Define the linear relation $J_0(A)$ by \eqref{Bailando}.
Then it is clear that $J_0(A) \subset A_0$ and
it follows from \eqref{ilme} that $\dom J_0(A)=\sX_0(A)$.
Since the sum \eqref{ilme} is direct, one sees that
the sum \eqref{EternalLoveAngelicaZapata} is direct.

It is clear that $A_S \oplus J_0(A) \subset A_0$. To see the reverse inclusion,
let $(x,y)\in A_0$.  Since $x\in\sR_0(A)$, it follows from \eqref{ilme}
that $x=x_1+x_2$ with $x_1\in \sX_0(A)=\dom J_0(A)$ and
$x_2\in\sR_c(A)$. Hence there exist elements $y_1\in\ran J_0(A)$
and $y_2\in\sR_c(A)$
such that $(x_1,y_1)\in J_0(A)$ and $(x_2,y_2)\in A_S$. With
$(x,y)\in A$ this gives $(0,y-y_1-y_2)\in A$, hence
\[
y-y_1-y_2\in \mul A \cap \sR_0(A) \subset\sR_c(A),
\]
so that $(0,y-y_1-y_2)\in A_S$.
It follows that
\[
(x,y)=(x_2, y_2) +(0,y-y_1-y_2) +(x_1,y_1)  \in A_S \oplus J_0(A),
\]
thus $A_0\subset A_S \oplus J_0(A)$.
Hence the identity \eqref{EternalLoveAngelicaZapata} has been shown.
The identity $J_{0}(A)=A\cap (\sX_0(A) \oplus \sX_0(A))$ is obvious from
\eqref{EternalLoveAngelicaZapata}.

 By construction,
$J_0(A)$ is an operator in $\sX_0(A)$.
 Assume that $(x,\la x)\in J_0(A)$ for some $\la\not=0$.
Then one sees
\[
x\in \sR_\la(A)\cap \sR_0(A)=\sR_c(A)
\]
by~\eqref{eq:Rc-lambda-mu}. Hence $x\in\sR_c(A) \cap\sX_0(A)=\{0\}$ by Step~4 and $\sigma_p (J_0(A))=\{0\}$ is proved. In particular, $J_0(A)$ is a Jordan operator in $\sX_0(A)$ corresponding to $0\in \dC$.

\medskip\noindent
\emph{Step 7}: It remains to show~\eqref{eq:dim-J0}, which directly follows from~\eqref{Henk2}.
\end{proof}

{ Before stating Theorem \ref{Thm:existenceAla}, some properties
of the shifted relation $A-\lambda$, $\lambda \in \dC$, will be discussed. Fix $\lambda\in\dC$. Then one has the obvious identities
\begin{equation}\label{olala}
\sR_\la(A) =\sR_0(A-\la),
\quad \sR_\infty(A-\lambda)=\sR_\infty(A).
\end{equation}
It is clear from \eqref{Alejandro} and  \eqref{olala} that
\[
 \sR_c(A-\lambda)=\sR_\lambda(A) \cap \sR_\infty(A),
\]
which, invoking~\eqref{eq:Rc-lambda-mu}, gives
\begin{equation}\label{olala+}
\sR_c(A)=\sR_c(A-\la).
\end{equation}
It is clear that
$\lambda \in \sigma_p(A)$ if and only if $0 \in \sigma_p(A-\lambda)$.
This equivalence can be refined:
\begin{equation*}
 \lambda \in \sigma_\pi(A) \quad \Leftrightarrow \quad 0 \in \sigma_\pi(A-\lambda);
\end{equation*}
 cf.\ \eqref{ook333} and \eqref{olala+}.
Hence, one obtains for the total root space (see \eqref{totalrootsp+}) that
\[
 \sR_r(A)=\sR_r(A-\lambda).
\]
As a consquence of \eqref{olala+} and  \eqref{srel}
the relation $A_S-\lambda$ is given by
 \[
\begin{split}
A\cap(\sR_c(A) \times \sR_c(A))-\la
&=(A-\lambda) \cap(\sR_c(A) \times \sR_c(A))\\
&=(A-\la)\cap(\sR_c(A-\la) \times \sR_c(A-\la)),
 \end{split}
\]
which leads to the identity
\begin{equation}\label{hern1}
A_S-\la=(A-\la)_S.
\end{equation}
Similarly, according to \eqref{olala} and \eqref{deux},
the relation $A_\lambda-\lambda$ is given by
\[
\begin{split}
A\cap (\sR_\la(A) \times \sR_\la(A))-\la
&=(A-\lambda)\cap (\sR_\la(A) \times \sR_\la(A)) \\
&=(A-\la)\cap (\sR_0(A-\la) \times \sR_0(A-\la)),
\end{split}
\]
which leads to the identity
\begin{equation}\label{herni2}
 A_\la -\la=(A-\la)_0.
\end{equation}
}

\medskip
{ For the case $\lambda \in \sigma_{\pi}(A) \cap \dC$
it will be shown that the linear relation $A_{\lambda}$,  given in \eqref{deux},
has the  reducing sum decomposition $A_{\lambda}=A_{S} \oplus J_{\lambda}(A)$
involving the completely singular part $A_{S}$ and a Jordan operator $J_{\lambda}(A)$.
 The sequence of quotient spaces $\sZ_k(A,\la)$ is defined by
\begin{equation}\label{vv0}
\begin{split}
&\sZ_1(A,\lambda):=\frac{\ker (A-\la)+\sR_c(A)}{\sR_c(A)},  \\
& \sZ_k(A,\lambda):=\frac{\ker (A-\lambda)^{k}+\sR_c(A)}{\ker (A-\la)^{k-1}+\sR_c(A)},
\quad k\geq 2.
\end{split}
\end{equation}
Since the denominator is included in the numerator,
each quotient space $\sZ_k(A)$, $k \geq 1$, is well defined.
The \textit{Weyr characteristic} of $A$ with respect to the quotient spaces \eqref{vv0}
is defined as the sequence $(W_k(\lambda))_{k\ge 1}$, where
\begin{equation}\label{herni3}
W_k(\lambda) := \dim \sZ_k(A,\lambda).
\end{equation}
One sees from \eqref{vv0} and \eqref{olala+} that
\begin{equation}\label{henni3333}
\sZ_k(A,\lambda)=\sV_k(A-\lambda), \quad W_k(\lambda) =\dim \sV_k(A-\lambda), \quad \lambda \in \dC.
\end{equation}
Since $\sH$ is finite-dimensional, the number
\begin{equation}\label{sla}
s(\lambda)=\min\setdef{k\in\dN}{\ker (A-\la)^{k+1}+\sR_c(A)=\ker (A-\la)^{k}+\sR_c(A)}
\end{equation}
is well defined, as follows from \eqref{neerst} (with $A$ replaced by $A-\lambda$) and \eqref{olala+}.
The following theorem will be proved by means of  Lemma~\ref{nullketten} via a shift.}

\begin{theorem}\label{Thm:existenceAla}
Let $A$ be a linear relation in a finite-dimensional space $\sH$ with $\la \in \sigma_{\pi}(A) \cap \dC$, and let  $s(\lambda) \geq 1$ be given by \eqref{sla}.
Then the Weyr characteristic
$(W_k(\la))_{k\geq 1}$ in \eqref{herni3} satisfies
\[
    W_1(\la) \ge W_2(\la) \ge \cdots \ge W_{s(\la)}(\la) \ge 1\quad \mbox{and}\quad W_k(\la)=0, \quad k>s(\lambda).
\]
Moreover, there exist Jordan chains for $A$
corresponding to~$\lambda$ of the form:
\begin{equation*}
\begin{array}{rl}
 (x_{s(\lambda)}^i, x_{s(\lambda)-1}^i+\la x_{s(\lambda)}^i), \dots, (x_2^i, x_1^i+\la x_2^i),\ (x_1^i, \la x_1^i),
      &1 \leq i \leq W_{s(\lambda)}(\la), \\
 (x_{s(\lambda)-1}^i, x_{s(\lambda)-2}^i+\la x_{s(\lambda)-1}^i), \dots,  (x_1^i,
 \la x_1^i),         &W_{s(\lambda)}(\la)+1 \leq i \\
 & \qquad\qquad \leq W_{s(\lambda)-1}(\la), \\
\ddots\qquad \qquad\qquad \vdots\quad\ \ & \qquad\qquad \vdots \\
  (x_2^i, x_1^i+\la x_2^i),\ (x_1^i,  \la x_1^i), & W_3(\la)+1 \leq i  \leq W_2(\la), \\
  (x_1^i,  \la x_1^i), & W_2(\la)+1 \leq i \leq W_1(\la)
\end{array}
\end{equation*}
where
$\{[x_{k}^{1}],\ldots,[x_{k}^{W_{k}(\la)}]\}$ is a basis
of $\sZ_{k}(A,\la)$ in \eqref{vv0}, $1\le k\le s(\lambda)$, and, consequently,
the elements in the set $\setdef{x_k^i}{1\le i\le W_k(\la),\ 1\le k\le s(\lambda)}$
are linearly independent in $\sH$.
Then the linear space $\sR_{\lambda}(A)$ has the direct sum decomposition
\begin{equation}\label{ilme+}
\sR_\lambda(A)= \sR_c(A) \oplus \sX_\lambda(A),
\end{equation}
where the linear space $\sX_{\lambda}(A)$ is given by
\begin{equation}\label{Bailandoo--}
\sX_{\lambda}(A)=\spn \setdef{ x_k^i}{1\le i\le W_k(\lambda),\ 1\le k\le s(\lambda)}.
\end{equation}
Furthermore, with respect to \eqref{ilme+}, the graph restriction of $A$
to $\sR_\lambda(A) \times \sR_\lambda(A)$,
\[
A_{\lambda}=A \cap (\sR_\lambda(A) \times \sR_\lambda(A))
\]
has the reducing sum decomposition
\begin{equation}\label{EternalLoveAngelicaZapata+}
A_\lambda = A_S \oplus J_\lambda(A),
\end{equation}
where the linear relation $J_{\lambda}(A)=A\cap (\sX_\lambda(A) \times \sX_\lambda(A))$ admits the representation
 \begin{equation}\label{coro}
\begin{split}
 J_{\lambda}(A) &=\spn
\big\{
 (x_k^i, x_{k-1}^i+\la x_k^i),\,  \dots  \dots , (x_2^i, x_1^i+\la x_2^i),\ (x_1^i, \la x_1^i):\, \\
&\hspace{3.5cm} W_{k+1}(\lambda)+1 \leq i \leq W_{k}(\lambda), \ 1 \leq k \leq s(\lambda) \big\},
\end{split}
 \end{equation}
In fact, $J_\lambda(A)$ is a Jordan operator in $\sX_\lambda(A)$ corresponding to~$\lambda$ and the total dimension of $J_\lambda(A)$ is
\begin{equation}\label{eq:dim-Jlambda}
\dim J_\lambda(A) =W_{1}(\lambda)+W_{2}(\lambda)+ \ldots+W_{s(\lambda)}(\lambda).
\end{equation}
 \end{theorem}

\begin{proof}
The assumption $\lambda \in \sigma_{\pi}(A) \cap \dC$
implies that
$0\in \sigma_{\pi}(A -\lambda)$.  Now apply Lemma~\ref{nullketten}
where $A$, $v$, and $d_k$ are replaced by
$A-\lambda$, $s(\lambda)$, and $W_k(\lambda)$.
Then $s(\lambda) \geq 1$ and it is clear  that
$(W_k(\lambda))_{k\geq 1}$ in \eqref{herni3}
is nonincreasing with $W_k(\lambda)=0$  for $k>s(\lambda)$, as follows from \eqref{henni3333}.

Moreover, with the Jordan chains in  Lemma~\ref{nullketten}
interpreted for $A-\lambda$ at the eigenvalue $0$,
the present Jordan chains for $A$ at the eigenvalue $\lambda$ follow.
For this purpose recall that
\begin{equation*}\label{hernia}
 (u_n, u_{n-1}+\la u_n), \dots, (u_2, u_1+\la u_2),\ (u_1, \la u_1)
 \end{equation*}
is a (Jordan) chain for $A$ at   $\lambda$ if and only if
\begin{equation*}\label{hernib}
 (u_n, u_{n-1}), \ (u_{n-1}, u_{n-2}), \dots, (u_2, u_1),\ (u_1, 0)
  \end{equation*}
is a (Jordan) chain for $A-\lambda$ at   $0$.
The statement about the basis of $\sZ_k(A,\lambda)$ follows from  \eqref{herni3}.

\medskip
According to Lemma~\ref{nullketten},
$\sR_{0}(A-\lambda)$
has the direct sum decomposition
\begin{equation}\label{ilmel-}
\sR_0(A-\lambda)= \sR_c(A-\lambda) \oplus \sX_0(A-\lambda),
\end{equation}
and, with respect to \eqref{ilmel-}, the linear relation
$(A-\lambda)_{0}$ has the reducing sum decomposition
\begin{equation}\label{vier-}
(A-\lambda)_0= (A-\lambda)_S \oplus J_0(A-\lambda).
\end{equation}
Using \eqref{olala} and \eqref{olala+} in \eqref{ilmel-}
gives the direct sum decomposition \eqref{ilme+}, where $\sX_\lambda(A)=\sX_0(A-\lambda)$. Likewise, using \eqref{herni2}, \eqref{hern1}  in \eqref{vier-},
 one obtains
 \[
A_\lambda-\lambda=(A_S -\lambda)\oplus J_0(A-\lambda),
\]
or, in other words,  the reducing sum decomposition  \eqref{EternalLoveAngelicaZapata+}, where $J_{\lambda}(A)=J_0(A-\lambda)+\lambda$. Recall from Lemma~\ref{nullketten} that
\[
\sX_0(A-\lambda)=\spn \setdef{ x_k^i}{1\le i\le W_k(\lambda),\ 1\le k\le s(\lambda)},
\]
which gives \eqref{Bailandoo--}. Likewise, one has from Lemma~\ref{nullketten} that
\[
 J_{0}(A-\lambda)=(A-\lambda) \cap (\sX_0(A-\lambda) \oplus \sX_0(A-\lambda)),
\]
which  leads to
\[
J_\lambda(A)=A \cap (\sX_\lambda(A) \times \sX_\lambda(A)),
\]
as stated in the theorem.
Similarly,
 \[
\begin{split}
 J_{0}(A-\lambda)&=\spn
\big\{
 (x_k^i, x_{k-1}^i),\,  \dots  \dots , (x_2^i, x_1^i),\ (x_1^i, 0):\, \\
&\hspace{3cm} W_{k+1}(\lambda)+1 \leq i \leq W_{k}(\lambda), \ 1 \leq k \leq s(\lambda) \big\},
\end{split}
\]
 which gives \eqref{coro}.
 Finally, from Lemma~\ref{nullketten} one obtains
\[
\dom J_{0}(A-\lambda)=\sX_{0}(A-\lambda) \quad \mbox{and}
\quad \sigma_p (J_0(A-\lambda))=\{0\},
\]
 which shows that $J_\lambda(A)$ is a Jordan operator in $\sX_\lambda(A)$ corresponding to~$\lambda$. The formula~\eqref{eq:dim-Jlambda} directly follows from the representation~\eqref{coro}.
 \end{proof}

{ Before stating Theorem \ref{Thm:existenceAlaInfty}, some general properties
of the inverse $A^{-1}$ of a linear relation $A$ will be discussed.  It is easy to see that  for $\lambda \in \dC$ one has:
\begin{equation}\label{hennia-}
 \ker (A-\lambda)=\ker \left(A^{-1} -\frac{1}{\lambda} \right), \quad \lambda \neq 0,
 \quad \mbox{and}  \quad
  \ker A =\mul A^{-1},
\end{equation}
and it is therefore clear that
\begin{equation}\label{hennia-g}
\begin{split}
 &\lambda \in \sigma_p(A) \quad \iff \quad \frac{1}{\lambda} \in \sigma_p(A^{-1}), \quad \lambda \neq 0, \\
  &0 \in \sigma_p(A) \quad\iff \quad \infty \in \sigma_p(A^{-1}).
\end{split}
\end{equation}
In a more general context, one also has
\begin{equation}\label{henni00}
 \sR_\lambda(A)=\sR_{\frac{1}{\lambda}}(A^{-1}), \quad \lambda \neq 0,
 \quad \mbox{and} \quad
\sR_0(A)=\sR_\infty(A^{-1}),
 \end{equation}
see  \cite[Lem.~3.2]{BSTW}.
 It is clear from \eqref{henni00} that
 \begin{equation}\label{henni0}
\sR_c(A)=\sR_c(A^{-1}).
\end{equation}
Moreover, the equivalences in \eqref{hennia-g},
together with \eqref{henni00} and \eqref{henni0}, lead to
\[
\begin{split}
&\lambda  \in \sigma_{\pi}(A) \quad \iff \quad  \frac{1}{\lambda} \in \sigma_{\pi}(A ^{-1}), \quad \lambda \neq 0, \\
 &0\in \sigma_{\pi}(A) \quad \iff \quad \infty \in \sigma_{\pi}(A ^{-1}).
\end{split}
\]
Thus, it follows from the definition of the total root space (see \eqref{totalrootsp+})  that
\begin{equation}\label{nieuw}
{ \sR_r(A)=\sR_r(A^{-1}).}
\end{equation}
As a consequence of \eqref{henni0} and \eqref{srel},
the completely singular part  $(A^{-1})_{S}$ of $A^{-1}$ is given by
\[
 (A^{-1})_S =A^{-1} \cap (\sR_c(A^{-1})\times \sR_c(A^{-1}))
 =A^{-1} \cap (\sR_c(A)\times \sR_c(A)),
\]
which leads to the identity
\begin{equation}\label{henni1}
((A^{-1})_S)^{-1} = A \cap (\sR_c(A)\times \sR_c(A))=A_S.
\end{equation}
Similarly, according to \eqref{henni00} and \eqref{deux},
the relation $(A^{-1})_0$ is given by
\[
(A^{-1})_0=A^{-1} \cap (\sR_0(A^{-1}) \times \sR_0(A^{-1}))=A^{-1} \cap (\sR_\infty(A) \times \sR_\infty(A)),
\]
which leads to the identity
\begin{equation}\label{henni2}
((A^{-1})_0)^{-1}= A \cap (\sR_\infty(A) \times \sR_\infty(A))=A_\infty.
\end{equation}
}

\medskip
{ For the case $\infty \in \sigma_{\pi}(A)$
it will be shown below that the linear relation $A_{\infty}$
has the reducing sum decomposition  $A_{\infty}=A_{S} \oplus J_{\infty}(A)$
involving the completely singular part $A_{S}$ and a Jordan relation $J_{\infty}(A)$.
The sequence of quotient spaces $\sW_k(A)$ is defined by
\begin{equation}\label{grijp0}
\sW_1(A):=\frac{\mul A+\sR_c(A)}{\sR_c(A)}, \quad
 \sW_k(A):=\frac{\mul A^{k}+\sR_c(A)}{\mul A^{k-1}+\sR_c(A)},\quad k\geq 2.
\end{equation}
Since the denominator is included in the numerator,
each quotient space $\sW_k(A)$, $k \geq 1$, is well defined.
The \textit{Weyr characteristic} of $A$ with respect to the quotient spaces
\eqref{grijp0}
is defined as the sequence $(A_k)_{k\ge 1}$, where
\begin{equation}\label{grijp1}
A_k := \dim \sW_k(A), \quad k \geq 2.
\end{equation}
One sees from \eqref{hennia-} and \eqref{henni0}  that
\begin{equation}\label{henni3}
 \sW_k(A)=\sV_k(A^{-1}), \quad A_k = \dim \sW_k(A)=  \dim \sV_k(A^{-1}).
\end{equation}
Since $\sH$ is finite-dimensional, the number
\begin{equation}\label{alef}
\aleph =\min\setdef{k\in\dN}{\mul A^{k+1}+\sR_c(A)=\mul A^{k}+\sR_c(A)}
\end{equation}
is well defined, as follows from \eqref{neerst} (with $A$ replaced by $A^{-1}$) and \eqref{henni0}.
The following theorem will be proved by
means of Lemma~\ref{nullketten} via an inversion.}

\begin{theorem}\label{Thm:existenceAlaInfty}
Let $A$ be a linear relation in a finite-dimensional space $\sH$ with
$\infty \in \sigma_{\pi}(A)$, and let $\aleph \geq 1$ be given by \eqref{alef}.
Then the Weyr characteristic $(A_k)_{k\geq 1}$ in \eqref{grijp1}
satisfies
\[
 A_1 \ge A_2 \ge \cdots \ge A_{\aleph} \ge 1\quad \mbox{and}\quad A_k=0, \quad k>\aleph.
\]
Moreover,
there exist Jordan chains for $A$ corresponding to the eigenvalue
$\infty$ of the form
\begin{equation*}
\begin{array}{ll}
(0, x_1^i),\ (x_1^i,x_2^i),\ \dots, (x_{\aleph-2}^i, x_{\aleph -1}^i),\ (x_{\aleph -1}^i, x_\aleph^i), &1 \leq i \leq A_\aleph,\\
(0,x_{1}^i),\ (x_1^i, x_2^i),\ \dots,  (x_{\aleph-2}^i,  x_{\aleph-1}^i), &A_\aleph+1 \leq i \leq A_{\aleph -1},\\
\ \ \quad \vdots \quad\qquad \vdots \qquad \iddots & \qquad \qquad\ \vdots \\
  (0, x_1^i),\ (x_1^i, x_2^i), & A_3+1 \leq i  \leq A_2, \\
  (0, x_1^i), & A_2+1 \leq i \leq A_1,
\end{array}
\end{equation*}
where
$\{[x_{k}^{1}],\ldots,[x_{k}^{A_{k}}]\}$ is a basis  of $\sW_{k}(A)$ in~\eqref{grijp0}, $1\le k\le \aleph$,
and, consequently, the elements in the set $\setdef{x_{k}^i}{1\le i\le A_k,\ 1\le k\le \aleph}$
are linearly independent in~$\sH$.
Then the linear space $\sR_\infty(A)$ has the direct sum decomposition
\begin{equation}\label{69TrollzMinaj}
\sR_\infty(A)= \sR_c(A) \oplus \sX_\infty(A),
\end{equation}
where the linear space $\sX_\infty(A)$ is given by
\begin{equation}\label{Bailandoo-}
 \sX_\infty(A)=\spn \setdef{ x_k^i}{1\le i\le A_k,\ 1\le k\le \aleph}.
\end{equation}
Furthermore, with respect to \eqref{69TrollzMinaj}, the graph restriction of $A$ to $\sR_\infty(A) \times \sR_\infty(A)$,
\[
A_\infty=A \cap (\sR_\infty(A) \times \sR_\infty(A))
\]
has the reducing sum decomposition
\begin{equation}\label{grii}
A_\infty = A_S \oplus J_{\infty}(A),
\end{equation}
where the linear relation $J_\infty(A)=A \cap (\sX_\infty(A) \times \sX_\infty(A))$ is given by
\begin{equation}\label{Bailandoo}
\begin{split}
 J_{\infty}(A)&=\spn
\big\{(0, x_1^i), \ (x_1^i, x_2^i), \, \dots  \dots , (x_{k-1}^i, x_{k}^i) :\, \\
 &\hspace{4cm} A_{k+1}+1 \leq i \leq A_{k}, \ 1 \leq k \leq \aleph \big\}.
\end{split}
\end{equation}
 In fact, $J_\infty(A)$ is a Jordan relation in $\sX_\infty(A)$ corresponding to~$\infty$ and the total dimension of $J_\infty(A)$ is
\begin{equation}\label{eq:dim-Jinfty}
\dim J_\infty(A) =A_{1}+A_{2}+ \ldots+A_{\aleph}.
\end{equation}
 \end{theorem}

\begin{proof}
The assumption $\infty \in \sigma_{\pi}(A)$ implies that
$0\in \sigma_{\pi}(A ^{-1})$.  Now apply Lemma~\ref{nullketten}
where $A$, $v$, and $d_k$ are replaced by
$A^{-1}$, $\aleph$, and $A_k$.
Then $\aleph \geq 1$ and it is clear  that
$(A_k)_{k\geq 1}$ in \eqref{grijp1}
is nonincreasing with $A_k=0$  for $k>\aleph$, as follows from \eqref{henni3}.

\medskip
{ Moreover, with the Jordan chains in  Lemma~\ref{nullketten}
interpreted for $A^{-1}$ at the eigenvalue $0$,
the present Jordan chains for $A$ at the eigenvalue $\infty$ follow.
For this purpose recall that
 \begin{equation*}\label{hennia}
 (u_n, u_{n-1}),\ (u_{n-1}, u_{n-2}),\ \dots \dots  \dots , (u_2, u_1),\ (u_1, 0)
\end{equation*}
is a Jordan chain for $A$ at $0$ if and only if
\begin{equation*}\label{hennib}
  (0,u_1), \ (u_1,u_2), \ \dots \dots \dots , (u_{n-2}, u_{n-1}, \ (u_{n-1},u_n)
\end{equation*}
is a Jordan chain for $A^{-1}$ at $\infty$.
The statement about the basis of $\sW_k(A)$
follows from  \eqref{henni3}.}

\medskip
According to Lemma~\ref{nullketten},
$\sR_{0}(A^{-1})$
has the direct sum decomposition
\begin{equation}\label{ilmel}
\sR_0(A^{-1})= \sR_c(A^{-1}) \oplus \sX_0(A^{-1}),
\end{equation}
and, with respect to \eqref{ilmel}, the linear relation
$(A^{-1})_{0}$ has the reducing sum decomposition
\begin{equation}\label{vier}
(A^{-1})_0= (A^{-1})_S \oplus J_0(A^{-1}).
\end{equation}
Using \eqref{henni00} and \eqref{henni0} in \eqref{ilmel} gives the direct sum decomposition \eqref{69TrollzMinaj}, where $\sX_\infty(A)=\sX_0(A^{-1})$. Likewise, using \eqref{henni1}, \eqref{henni2}, and  taking inverses in \eqref{vier}, one obtains
the reducing sum decomposition  \eqref{grii}, where $J_{\infty}(A)=(J_0(A^{-1}))^{-1}$. Recall from Lemma~\ref{nullketten} that
\[
\sX_0(A^{-1})=\spn \setdef{ x_k^i}{1\le i\le A_k,\ 1\le k\le \aleph},
\]
which gives \eqref{Bailandoo-}. Likewise, one has from Lemma~\ref{nullketten}
\[
 J_{0}(A^{-1})=A^{-1}\cap (\sX_0(A^{-1}) \oplus \sX_0(A^{-1})),
\]
which, taking inverses, leads to
\[
J_\infty(A)=A \cap (\sX_\infty(A) \times \sX_\infty(A)),
\]
as stated in the theorem.
Similarly,
 \[
\begin{split}
 J_{0}(A^{-1})&=\spn
\big\{
 (x_k^i, x_{k-1}^i),\,  \dots  \dots , (x_2^i, x_1^i),\ (x_1^i, 0):\, \\
&\hspace{4cm} A_{k+1}+1 \leq i \leq A_{k}, \ 1 \leq k \leq \aleph \big\}.
\end{split}
\]
 Taking the inverse of the linear relations on both sides of the above equation
gives~\eqref{Bailandoo}.
Note that $J_0(A^{-1})$ is an operator and that $\ker J_\infty(A)=\mul J_0(A^{-1})=\{0\}$;
hence the relation $J_\infty(A)$ is injective.
Finally, recall
from Lemma~\ref{nullketten} that
\[
\dom J_{0}(A^{-1})=\sX_{0}(A^{-1}) \quad \mbox{and} \quad \sigma_p (J_0(A^{-1}))=\{0\},
\]
by which $J_\infty(A)$ is a Jordan relation in $\sX_\infty(A)$ corresponding to~$\infty$. The formula~\eqref{eq:dim-Jinfty} directly follows from the representation~\eqref{Bailandoo}.
\end{proof}

Theorems~\ref{Thm:existenceAla} and \ref{Thm:existenceAlaInfty} will now be combined
to show the main result of this section: a reducing sum decomposition of the root part.

\begin{theorem}\label{Thm:DecompAS-AJ}
Let $A$ be a linear relation in a finite-dimensional space $\sH$
and let $\sigma_{\pi}(A) \setminus \{\infty\}= \{\lambda_1,\ldots,\lambda_l\}$.
For all $\la\in\si_{\pi}(A)\setminus \{\infty\}$ let $J_\la(A)$ be the Jordan operator in $\sX_{\lambda}(A)$ corresponding to $\lambda$
such that $A_\la = A_S \oplus  J_\la(A)$ is the reducing sum decomposition
as in Theorem~{\rm \ref{Thm:existenceAla}}
and, if $\infty \in \sigma_{\pi}(A)$, let $J_\infty(A)$ be the Jordan relation in $\sX_{\infty}(A)$
such that $A_\infty = A_S \oplus  J_\infty(A)$ is the reducing sum decomposition
as in Theorem~{\rm \ref{Thm:existenceAlaInfty}}.
 Then $\sR_r(A)$ has the direct sum decomposition
\begin{equation}\label{direkt}
\sR_r(A)=\sR_c(A) \oplus \sX_{\la_1}(A) \oplus \cdots \oplus \sX_{\la_l}(A)\oplus \sX_\infty(A).
\end{equation}
Furthermore, with respect to \eqref{direkt},
the linear relation $A_R$ has the reducing sum decomposition
\begin{equation}\label{LanadelRey}
A_R = A_S \oplus J_{\la_1}(A)\oplus \cdots\oplus J_{\la_l}(A)\oplus J_\infty(A).
\end{equation}
Furthermore, one has the equalities
\begin{equation}\label{un}
\sR_{r}(A)= \dom A_{R} +\ran J_{\infty}(A) = \ran A_{R} +\dom J_{0}(A).
\end{equation}
If $\infty \notin \sigma_{\pi}(A)$, then the space $\sX_\infty(A)$
and the linear relation $J_\infty(A)$ are absent.
\end{theorem}

\begin{proof}
First, it will be shown that the identity  \eqref{direkt} holds and that the sum is direct.
Recall from Theorem~\ref{Thm:existenceAla} that $\sR_{\la_i}(A)=\sR_c(A)\oplus \sX_{\la_i}(A)$
for $1\le i\le l$, and $\sR_{\la}(A)=\sR_c(A)$ for $\la\notin \sigma_{\pi}(A)$ by~\eqref{ook333}.
Hence, the following identity is clear:
\begin{equation*}
\sR_r(A)=\sR_c(A)+\sX_{\la_1}(A) + \cdots + \sX_{\la_l}(A)+ \sX_\infty(A).
\end{equation*}
To see that the sum on the right-hand side is direct, let the elements
$x_c\in\sR_c(A)$, $x_i\in \sX_{\la_i}(A)$, $1\le i\le l$, and
$x_\infty \in \sX_\infty(A)$ be such that
\[
x_c+x_1 + \cdots + x_l +x_\infty =0.
\]
Then it follows from \cite[Cor.~4.5]{BSTW} that
\[
x_i\in \sR_{\la_i}(A) \cap \left(\sR_c(A)+\sR_\infty(A)+\sum_{j=1,j\not=i}^l \sR_{\la_j}(A)\right)=\sR_c(A),\quad 1\le i\le l,
\]
hence $x_i\in \sX_{\la_i}(A) \cap \sR_c(A)=\{0\}$. In a similar way one obtains
$x_\infty =0$ and concludes that  $x_c=0$. Thus the sum in~\eqref{direkt}  is direct.

\medskip\noindent
 Next, it will be shown that the identity  \eqref{LanadelRey} holds and that the sum is direct.
Observe that it follows from \eqref{direkt}
that the sum
\[
A_S \oplus J_{\la_1}(A) \oplus \cdots \oplus J_{\la_l}(A) \oplus J_\infty(A)
\]
is direct. From  Theorem~\ref{Thm:existenceAla} and Theorem~\ref{Thm:existenceAlaInfty} one finds that
\[
A_S \oplus J_{\la_1}(A) \oplus \cdots \oplus J_{\la_l}(A) \oplus J_\infty(A) \subset A_R.
\]
 In order to show the reverse inclusion, let $(x,y)\in A_R$.
 Then $x\in \sR_r(A)$
and, by~\eqref{direkt},  there exist $x_c\in \sR_c(A)$,
$x_i \in \dom J_{\lambda_i}(A)$, $1 \leq i \leq l$, and $x_{\infty} \in \ran J_{\infty}(A)$ such that
\[
x=x_c+x_1+ \cdots + x_l+x_\infty.
\]
Hence,  there exist $y_c\in \sR_c(A)$ with $(x_c,y_c)\in A$ and
$y_i\in\ran J_{\lambda_i}(A)$ with $(x_i,y_i)\in A$ for $1\le i\le l$ so that
\[
(x,y)= (x_c,y_c) + (x_1,y_1) + \cdots + (x_l,y_l) + (x_\infty , \tilde y := y- y_c-y_1-\cdots-y_l) \in A.
\]
In particular, one has
\[
(x_c,y_c) \in A_S, \quad (x_i,y_i) \in J_{\lambda_i}, \ 1 \leq i \leq l,\quad \mbox{ and }\quad (x_\infty , \tilde y) \in A.
\]
As $x_\infty \in \sR_\infty(A)$ it follows that $\tilde y \in \sR_\infty(A)$. By \eqref{69TrollzMinaj} there exist $\tilde y_\infty \in \sX_\infty(A)$ and $\hat y \in \sR_c(A)$
 such that $\tilde y = \tilde y_\infty + \hat y$, and $\tilde x_\infty \in \sX_\infty(A)$ such that $(\tilde x_\infty, \tilde y_\infty)\in J_\infty(A) \subset A$. Then
\[
    (x,y) = (x_c,y_c) + (x_1,y_1) + \cdots + (x_l,y_l) + (\tilde x_\infty, \tilde y_\infty) + (x_\infty - \tilde x_\infty,\hat y) \in A
\]
and one sees that the last term belongs to $A$.
 But since $x_\infty - \tilde x_\infty\in \sX_\infty(A)$
and $\hat y \in \sR_c(A)$ it follows that
$x_\infty - \tilde x_\infty \in \sR_c(A) \cap \sX_\infty(A) = \{0\}$, thus $x_\infty = \tilde x_\infty$.
Therefore, $(x,y)$ belongs to the right-hand side of \eqref{LanadelRey}.

Finally,  equation~\eqref{un} directly follows from Theorems~\ref{Blaubeere}, ~\ref{Thm:existenceAla} and~\ref{Thm:existenceAlaInfty}.
\end{proof}

Recall from \eqref{domran} that the restriction of $A$ to $\sH(A)=\dom A + \ran A$
is formally the same relation but in a possibly smaller space.
In particular the space $\sR_{r}(A)$ is contained in $\sH(A)$:
\[
 \sR_{r}(A) \subset \dom A + \ran A.
\]
As a consequence of Theorem \ref{Thm:DecompAS-AJ} it is possible to
characterize when equality holds.
The case of strict inclusion will be considered in detail
in  Section \ref{Sec:Multishift}.

\begin{corollary}\label{greijp}
Let $A$ be a linear relation in a finite-dimensional space $\sH$.
Then the following statements are equivalent:
\begin{enumerate}\def\labelenumi{\rm(\roman{enumi})}
\item $\dom A+\ran A = \sR_r(A)$,
\item $\dom A \subset \sR_r(A)$,
\item $\ran A \subset \sR_r(A)$.
\end{enumerate}
\end{corollary}

\begin{proof}
(i) $\Rightarrow$ (iii) is clear.

(iii) $\Rightarrow$ (i): Since $\ran A \subset \sR_{r}(A)$,
it suffices to show that $\dom A \subset  \sR_r(A)$.
Assume that $x\in\dom A$, then $(x,y)\in A$ for some $y\in\ran A \subset \sR_r(A)$.
Then $y=y_r+y_0$ with $y_r\in\ran A_R$ and $y_0\in\dom J_0(A)$
according to \eqref{un}, hence $(x_r,y_r)\in A_R$ for some $x_r\in \sR_r(A)$.
The decomposition
\[
(x,y)=(x_r,y_r)+(x-x_r,y_0)
\]
shows that $(x-x_r,y_0)\in A$. Due to $y_0\in\dom J_0(A) \subset \ker A^{i}$ for some $i$ one finds $x-x_{r} \in \ker A^{i+1}\subset \sR_{r}(A)$.
Hence, it follows that
$x\in\sR_r(A)$.

(ii) $\Leftrightarrow$ (iii): This is due to the symmetry when $A$ is replaced by $A^{-1}$; cf. \eqref{nieuw}.
 \end{proof}

\section{The multishift part of a linear relation}\label{Sec:Multishift}

Let $A$ be a linear relation in a finite-dimensional space $\sH$.
In this section it will be shown that
there exists a linear subspace $\sR_m(A)\subseteq \sH$
spanned by entries of linearly independent shift chains
such that the restriction of $A$ to $\sR_m(A) \times \sR_m(A)$ is a multishift.

\medskip

The construction of the shift chains in $A$ is based on an
appropriate choice of a sequence of quotient spaces.
 The sequence of quotient spaces $\sM_k(A)$ is defined by
\begin{equation}\label{eq:Dk}
\sM_0(A):=\frac{\dom A +\ran A}{\ran A+\sR_r(A)}, \qquad
\sM_k(A):=\frac{\ran A^k+\sR_r(A)}{\ran A^{k+1}+\sR_r(A)},\quad k\geq 1.
\end{equation}
Indeed, as the denominator is included in the numerator, each quotient space $\sM_k(A)$,
$k \geq 0$,
is well defined.
The \textit{Weyr characteristic} of $A$ with respect to \eqref{eq:Dk}
is defined as the sequence $(C_{k})_{k \geq 0}$, where
\begin{equation}\label{AdanYEva}
C_k:=\dim \sM_k(A), \quad k\geq 0.
\end{equation}
If $\dom A+\ran A = \sR_r(A)$, then by Corollary~\ref{greijp} this condition is equivalent
to $\ran A \subset \sR_{r}(A)$.
This implies that $\ran A^{k} \subset \sR_{r}(A)$ for all $k \geq 1$,
so that $C_{k}=0$ for all $k \geq 0$. 
In this case one may define $\sR_{m}(A):=\{0\}$ and then the restriction
$A_{M}$ is given by
$A_{M}=A \cap (\sR_{m}(A)  \times \sR_{m}(A))=\{0,0\}$.

\medskip

Now consider the case that the inclusion $\sR_{r}(A) \subset \dom A+\ran A $
is strict. Then the sequence in \eqref{AdanYEva} is not trivial, although ultimately
the entries are zero. To see this, observe that since
 the linear space $\sH$ is finite-dimensional, the
number
\begin{equation}\label{eq:d_mult}
m=\min \setdef{k\in\dN}{\ran A^{k+1}+\sR_r(A)=\ran A^{k}+\sR_r(A)}
\end{equation}
is well defined.

\begin{lemma}\label{Caramelo}
Let $A$ be a linear relation in a finite-dimensional space $\sH$ and assume
that the inclusion $\sR_r(A) \subset \dom A+\ran A$  is strict.  Then the Weyr characteristic $(C_{k})_{k \geq 0}$ is nontrivial.
In fact, with $m$ given by \eqref{eq:d_mult}, one has
\begin{equation}\label{wiederzuruck}
 \ran A^{k}+\sR_r(A)= \ran A^{k+1}+\sR_r(A),   \quad k \geq m,
 \end{equation}
and it follows that $C_{k}=0$, $k \geq m$. Moreover,
 \begin{equation}\label{shoor}
 \ran A^k \subset \sR_r(A), \quad k \geq m,
 \end{equation}
which, in particular, implies that $m \geq 2$ and, consequently, $C_{1} \geq 1$.
\end{lemma}

\begin{proof}
The proof is divided into a number of steps.

\medskip\noindent
\textit{Step 1}: First observe that
\begin{equation}\label{raran}
\sR_r(A)+\ran A^k=\sR_0(A)+\ran A^k, \quad k\ge 1.
\end{equation}
To prove \eqref{raran} it suffices to show that the left-hand side
is contained in the right-hand side.
Recall from \eqref{direkt}, that
\begin{equation*}
\sR_r(A)=\sR_{\not=0}(A)\oplus \sX_{0}(A) \quad \mbox{where}
\quad  \sR_{\not=0}(A)=\sR_c(A) \oplus \sum_{\la\in\si_\pi \setminus \{0\}}\sX_{\la}(A) .
\end{equation*}
Fix $k \geq 1$. Since $\sR_c(A)$ is spanned by entries of singular chains~\eqref{schain}, it is clear that $\sR_c(A) \subset \ran A^k$.
Moreover, for $\la\in\si_\pi(A) \setminus \{0,\infty\}$ the Jordan
 operator $J_\lambda(A)$ defined in Theorem~\ref{Thm:existenceAla} is a bijection in $\sX_\lambda(A)$,
from which it follows that $\ran J_\lambda(A)=\sX_\lambda(A)$. Hence,
\[
 \sX_\lambda(A)=\ran (J_\lambda(A))^k \subset \ran A^k,
\]
where the last inclusion follows from $J_\lambda(A)\subset A$.
For the Jordan relation $J_\infty(A)$ defined in Theorem~\ref{Thm:existenceAlaInfty} it is immediate that $\ran J_\infty(A)=\sX_\infty(A)$, so that
\[
 \sX_\infty(A)=\ran (J_\infty(A))^k \subset \ran A^k,
\]
where the last inclusion is due to $J_\infty(A) \subset A$.
Therefore one finds that
\begin{equation*}\label{reduce}
\sR_{\not=0}(A)\subset \ran A^k, \quad k\ge 1.
\end{equation*}
Hence \eqref{raran} has been shown.

\medskip\noindent
\textit{Step 2}:
In order to prove \eqref{wiederzuruck},
 by \eqref{raran} it is sufficient to show that
\begin{equation}\label{kakaone}
\ran A^{k}+\sR_0(A)= \ran A^{k+1}+\sR_0(A), \quad k \geq m.
\end{equation}
By induction it suffices  to show that
$\ran A^{k}+ \sR_0(A)=\ran A^{k+1}+ \sR_0(A)$  for some $k \geq m$
implies that
$\ran A^{k+1}+ \sR_0(A)=\ran A^{k+2}+ \sR_0(A)$.
It will be shown that
\begin{equation}\label{kkakaone}
\ran A^{k+1}+ \sR_0(A)\subset \ran A^{k+2}+ \sR_0(A),
\end{equation}
since the converse inclusion follows from the inclusion
$\ran A^{k+2}\subset \ran A^{k+1}$.
Let $x\in \ran A^{k+1}+ \sR_0(A)$.
Then $x = x_r + x_0$ for some $x_r\in \ran A^{k+1}$ and $x_0\in\sR_0(A)$.
Furthermore, there exists $x_k\in \ran A^k$ such that $(x_k,x_r)\in A$.
By assumption one has
\[
x_k=x_{k+1}+z_0 \quad \mbox{with} \quad x_{k+1}\in \ran A^{k+1}
\quad \mbox{and} \quad z_0\in \sR_0(A).
\]
and hence there exists $y_0\in\sR_0(A)$ with $(z_0,y_0)\in A$. With
$(x_{k+1}+z_0, x_r)\in A$ it follows that $(x_{k+1}, x-x_0-y_0)\in A$ and
$x-x_0-y_0\in\ran A^{k+2}$. Hence one obtains $x\in\ran A^{k+2}+ \sR_0(A)$.
This shows \eqref{kkakaone}.

\medskip\noindent
\textit{Step 3}: To prove \eqref{shoor},
 it suffices to show that $\ran A^m \subset \sR_r(A)$, since
 $\ran A^k \subset \ran A^m$ for $k\ge m$.
Let $x\in \ran A^{m}$, then $x=x_1+x_1^0$ with
$x_1\in \ran A^{m+1}$  and $x_1^0\in\sR_0(A)$ by~\eqref{kakaone}.
Hence, there exists
$y_1\in\ran A^{m}$ with $(y_1,x_1)\in A$. Again by~\eqref{kakaone},
$y_1=x_2+x_2^0$ with
$x_2\in \ran A^{m+1}$  and $x_2^0\in\sR_0(A)$, and
 there is some
$x_2^1\in\sR_0(A)$ such that $(x_2^0, x_2^1)\in A$.
With $(x_2+x_2^0, x_1)=(y_1,x_1) \in A$
it follows that $(x_2, x_1-x_2^1)\in A$.

Next observe that there exists $y_2\in\ran A^{m}$ with $(y_2,x_2)\in A$, and by \eqref{kakaone},
$y_2=x_3+x_3^0$ with
$x_3\in \ran A^{m+1}$  and $x_3^0\in\sR_0(A)$. Moreover, there are
$x_3^1\in\sR_0(A)$ and $x_3^2\in\sR_0(A)$ such that $(x_3^0, x_3^1)\in A$ and
$(x_3^1, x_3^2)\in A$.  With $(x_3+x_3^0, x_2)=(y_2,x_2)\in A$
it follows that $(x_3, x_2-x_3^1)\in A$, and
$$
(x_3, x_2-x_3^1),\quad (x_2-x_3^1, x_1-x_2^1-x_3^2)
$$
form a chain in $A$. A continuation of this argument leads to a chain of the form
 $$
(z_n, z_{n-1}), \ldots , (z_3, z_2), (z_2, z_1)\in A,
$$
where each $z_k$ is of the form $z_k=x_k+z_k^0$ with $x_k\in \ran A^{m+1}$
and $z_k^0\in\sR_0(A)$.
Let $l \geq 2$ be the smallest index such that $z_1,\ldots, z_{l-1}$ are
linearly independent and $z_l\in\spn\left\{z_1,\ldots, z_{l-1}\right\}$.
Let $B=\spn\left\{(z_k, z_{k+1}): 1\le k\le l-1\right\}$.
 Then $B$ is an everywhere defined linear operator in
$\spn\left\{z_1,\ldots, z_{l-1}\right\} = \sR_r(B)$ with $B \subset A^{-1}$.
By means of \eqref{nieuw} one has that
$\sR_r(B) \subset \sR_r(A^{-1}) =\sR_r(A)$,
thus  $z_1\in \sR_r(A)$.
Hence, $x=x_1+x_1^0 = z_1-z_1^0 + x_1^0 \in\sR_r(A)$ follows.

\medskip\noindent
\textit{Step 4}: If $m=1$, then it follows from \eqref{shoor} that
$\ran A \subset \sR_{r}(A)$.
By Corollary~\ref{greijp} this contradicts the assumption that the inclusion
$\sR_r(A) \subset \dom A+\ran A$  is strict.
Thus $m \geq 2$ and, in particular,
$\ran A^{2} +\sR_{r}(A)$ is a proper subset of $\ran A +\sR_{r}(A)$,
which implies that $C_{1} \geq 1$.
\end{proof}
}

\begin{theorem}\label{Thm:DecompAR-AM}
Let $A$ be a linear relation in a finite-dimensional space $\sH$ and assume
that the inclusion $\sR_r(A) \subset \dom A+\ran A$  is strict, so that
$m$  in \eqref{eq:d_mult} satisfies $m \geq 2$. Then
the Weyr characteristic $(C_k)_{k\geq 0}$ in \eqref{AdanYEva} satisfies
\begin{equation*}\label{brauna}
 C_{0} = C_{1} \geq \cdots  \geq C_{m-1} \geq 1 \quad \mbox{and} \quad C_{k}=0,
 \quad k \geq m.
\end{equation*}
Moreover, there exist shift chains for $A$ of the following form
\begin{equation}\label{Henk3}
 \begin{array}{ll}
(x^i_0, x_{1}^i),\,
 (x_1^i, x_{2}^i),\,   \dots  ,
  (x_{m-3}^i, x^i_{m-2}),\ (x_{m-2}^i, x^i_{m-1}), & \quad\quad\quad \ \ 1 \leq i \leq C_{m-1}, \\
(x^i_0, x_{1}^i),\,
 (x_1^i, x_{2}^i),\,  \dots , (x_{m-3}^i, x^i_{m-2}), &C_{m-1}+1 \leq i \leq C_{m-2}, \\
\ \ \quad \vdots \quad\qquad \vdots \qquad \iddots & \quad \qquad \qquad\ \vdots \\
 (x^i_0, x_{1}^i),\,
 (x_1^i, x_{2}^i), & \quad \ C_3+1 \leq i \leq C_2,\\
 (x^i_0, x_{1}^i), & \quad \ C_2+1 \leq i \leq C_1,\\
\end{array}
\end{equation}
where
$\{[x_{k}^{1}],\ldots,[x_{k}^{C_{k}}]\}$ is a basis  of $\sM_{k}(A)$, $0\leq k\leq m-1$. The elements in $\setdef{x_k^i}{1\le i\le C_k,\ 0\le k\leq m-1}$ are linearly independent in $\sH$. Then $\sH(A) = \dom A+\ran A$ has the direct sum decomposition
\begin{equation}\label{50cent}
\sH(A)=\sR_r(A) \oplus \sR_m(A),
\end{equation}
where
\begin{equation}\label{lilnasX}
\sR_m(A):=\spn \setdef{ x_k^i}{1\le i\le C_k,\ 0\le k\le m-1}.
\end{equation}
Furthermore, with respect to the decomposition \eqref{50cent},
the relation $A$  has the reducing sum decomposition
\begin{equation}\label{9Euro}
A=A_R \oplus A_M,
\end{equation}
where $A_M:= A\cap (\sR_m(A) \times \sR_m(A))$ admits the representation
\begin{equation}\label{eq:ChainStrucAM}
\begin{split}
 A_M&=\spn
\big\{
 (x^i_0, x_{1}^i),\,
 (x_1^i, x_{2}^i),\,   \dots  \dots , (x_{k-1}^i, x^i_{k}):\, \\
&\hspace{5cm} C_{k+1}+1 \leq i \leq C_{k}, \ 1 \leq k \leq m-1\big\}.
\end{split}
\end{equation}
In fact, $A_M$ is a multishift and the total dimension of $A_{M}$ is
\begin{equation}\label{eq:dim-AM}
\dim A_{M} =C_{1}+C_{2}+ \ldots+C_{m-2}+C_{m-1}.
\end{equation}
\end{theorem}

\begin{proof}
Throughout the proof the identity~\eqref{raran} will be used. The proof is carried out in several steps.

\medskip\noindent
\emph{Step 1}:
For $1\le k\le m-1$, define the linear relations
\begin{equation*}
\hat B_k :=  \setdef{ ([y],[x])\in \sM_{k}(A) \times \sM_{k-1}(A)}{
\exists\, (x',y')\in A \mbox{ with } [x']=[x]\mbox{ and } [y']=[y] }.
\end{equation*}
It is shown that $\hat B_k: \sM_{k}(A) \to \sM_{k-1}(A)$ are injective operators, i.e.,
$\dom \hat B_k=\sM_{k}(A)$ and $\ker \hat B_k=\mul\hat B_k=\{[0]\}$ for $1\le k\le m-1$.
Moreover, $\ran \hat B_1=\sM_0(A)$.

To see that $\dom\hat  B_k=\sM_{k}(A)$ let $[y]\in\sM_{k}(A)$, then $y=y_1+y_2$
with $y_1\in\ran A^k$ and $y_2\in\sR_{0}(A)$,
and there exists $x\in\ran A^{k-1}$ ($x\in\dom A$ if $k=1$)
such that $(x,y_1)\in A$. Since $y-y_1=y_2\in\sR_{0}(A)$ one has $[y]=[y_1]$
and as further $[x]\in\sM_{k-1}(A)$ it follows $([y],[x])=([y_1],[x])\in\hat B_k$.

To see that $\mul\hat B_k=\{[0]\}$ let $([0],[x])\in \hat B_k$.
Then there exist $y'\in[0]$ and $x'\in[x]$
such that $(x',y')\in A$. Hence there exist $y_1\in\ran A^{k+1}$ and $y_2\in\sR_{0}(A)$ with
$y'=y_1+y_2$. Therefore,
$(x_1,y_1)\in A$ with some $x_1\in\ran A^k$. It follows that
$(x'-x_1,y_2)\in A$, and since $y_2\in\sR_{0}(A)$ one obtains $x'-x_1\in\sR_{0}(A)$.
Consequently, $[x]=[x_1]$ and
$x_1\in\ran A^k + \sR_{0}(A)$ which gives $[x]=[0]\in\sM_{k-1}(A)$.

To see that $\ker \hat B_k=\{[0]\}$ let $([y],[0])\in \hat B_k$.
Then there exists $x'\in [0]$ and $y'\in[y]$ such that $(x',y')\in A$.
Hence there exist $x_1\in\ran A^k$ and $x_2\in\sR_{0}(A)$ with $x'=x_1+x_2$.
Furthermore, there exists $y_2\in\sR_{0}(A)$ with
$(x_2,y_2)\in A$.
It follows that $(x_1, y'-y_2)\in A$, which together with $x_1\in\ran A^k$ implies that
$y'-y_2\in\ran A^{k+1}$.
Consequently, $[y] = [y_2]$ and $y_2\in \ran A^{k+1} + \sR_{0}(A)$ which gives $[y]=[0]\in\sM_k(A)$.

To see that $\ran \hat B_1=\sM_0(A)$ let $[x]\in\sM_{0}(A)$. Then $x=x_1+x_2$ with $x_1\in\dom A$ and
$x_2\in\ran A$, and there exists $y_1$ with $(x_1,y_1)\in A$ so that $([y_1],[x_1])\in \hat B_1$.
Since $x-x_1\in\ran A$ it follows that $[x]=[x_1]$ and the statement of Step 1 is shown.

The properties of $\hat B_k$ imply that
$$
C_0=C_1,\quad C_{k-1}\ge C_k,\quad 2\le k\le m.
$$

\medskip\noindent
\emph{Step 2}:
Let $\{[x^1_{m-1}],\ldots,[x^{C_{m-1}}_{m-1}]\}$ be
a basis of $\sM_{m-1}(A)$. Then, for $i=1,\ldots,C_{m-1}$, $x_{m-1}^i=x_0^i+\tilde x_{m-1}^i$
with $x_0^i\in\sR_0(A)$ and $\tilde x_{m-1}^i\in\ran A^{m-1}$.
Therefore, there are elements $x_{m-2}^i\in\ran A^{m-2}$ with $(x_{m-2}^i,\tilde x_{m-1}^i)\in A$,
and $[x_{m-1}^i]=[\tilde x_{m-1}^i]$  for $i=1,\ldots,C_{m-1}$, thus
it is shown that $\{[\tilde x^1_{m-1}],\ldots,[\tilde x^{C_{m-1}}_{m-1}]\}$ is a basis of $\sM_{m-1}(A)$ and $[x_{m-2}^i]=\hat B_{m-1}[\tilde x_{m-1}^i]$ for $i=1,\ldots,C_{m-1}$. Since $\hat B_{m-1}$ is injective by Step~1,
the elements $[x_{m-2}^1],\ldots,[x^{C_{m-1}}_{m-2}]\in \sM_{m-2}(A)$ are linearly independent. Now choose additional linearly independent elements $[x_{m-2}^i]\in\sM_{m-2}(A)$, $C_{m-1}+1\le i\le C_{m-2}$ (note that this range is empty if $C_{m-1} = C_{m-2}$), with  $x_{m-2}^i\in\ran A^{m-2}$ (this can be achieved by substracting appropriate elements from $\sR_0(A)$ without changing the equivalence class) such that $\{[x_{m-2}^1] ,\ldots,
[x^{C_{m-2}}_{m-2}]\}$ forms a basis of $\sM_{m-2}(A)$.

To continue in an inductive way, assume that, for some $2\le k\le m-2$,
$\{[x^1_{k}],\ldots,[x^{C_{k}}_{k}]\}$ is
a basis of $\sM_{k}(A)$ such that $x^1_{k},\ldots,x^{C_{k}}_{k}\in \ran A^k$. Then there exist $x_{k-1}^i\in\ran A^{k-1}$ such that $(x_{k-1}^i, x_{k}^i)\in A$ for $i=1,\ldots,C_{k}$. Therefore, $[x_{k-1}^i]=\hat B_{k}[\tilde x_{k}^i]$ for $i=1,\ldots,C_{m-1}$ and, since $\hat B_{k}$ is injective by Step~1,
the elements $[x_{k-1}^1],\ldots,[x^{C_{k}}_{k}]\in \sM_{k-1}(A)$ are linearly independent. Choose additional linearly independent elements $[\tilde x_{k-1}^i]\in \sM_{k-1}(A)$ for $C_{k}+1\le i\le C_{k-1}$ with $x^{i}_{k-1}\in\ran A^{k-1}$ such that $\{[\tilde x^1_{k-1}],\ldots,[\tilde x^{C_{k-1}}_{k-1}]\}$ is
a basis of $\sM_{k-1}(A)$.

This procedure continues until one arrives at a basis $\{[x^1_{1}],\ldots,[x^{C_{1}}_{1}]\}$ of $\sM_{1}(A)$ with $x^1_{1},\ldots,x^{C_{1}}_{1}\in \ran A$. Then there are elements $x_0^i$ with
$(x^i_{0},x^i_{1})\in A$ for $i=1,\ldots,C_{1} = C_0$. Since $[x_{0}^i]=\hat B_{1}[\tilde x_{1}^i]$ for $i=1,\ldots,C_{0}$ and $\hat B_{1} : \sM_1(A) \to \sM_0(A)$ is bijective by Step~1, $\{[x^1_{0}],\ldots,[x^{C_{0}}_{0}]\}$ is a basis of $\sM_0(A)$. In the end, the shift chains as in the statement of the theorem have been constructed.

\medskip\noindent
\emph{Step 3}:
It follows from the construction in Step 2 that
$\{[x_{k}^{1}],\ldots,[x_{k}^{C_{k}}]\}$ is a basis  of $\sM_{k}(A)$, $0\leq k\leq m-1$.
 To see that the elements
$\setdef{x_k^i}{1\le i\le C_k,\ 0\le k\leq m-1}$
are linearly independent in $\sH$, assume that
\begin{equation}\label{umum}
 \sum_{k=0}^{m-1} \sum_{i=1}^{C_k} c_k^i x_k^i=0.
\end{equation}
By \eqref{raran} $\sum_{k=1}^{m-1} \sum_{i=1}^{C_k} c_k^i x_k^i \in\ran A + \sR_0(A)$, so that
by taking equivalence classes in \eqref{umum} with respect to $\sM_0(A)$, one obtains
\[
\sum_{i=1}^{C_0}c_0^i [x_0^i]=0 \in \sM_0(A),
\]
which implies that $c_0^i=0$ for $1 \leq i \leq C_0$. Note that therefore the assumption \eqref{umum}
is reduced to
\begin{equation*}\label{umum1}
 \sum_{k=1}^{m-1} \sum_{i=1}^{C_k} c_k^i x_k^i=0.
\end{equation*}
Now form equivalence classes in $\sM_{1}(A)$ and
proceed in a similar way. Then ultimately it follows that
$c_k^i=0$ for all the coefficients, which proves the claim.

\medskip\noindent
\emph{Step 4:}
To show~\eqref{50cent}, first observe that it is clear that $\sR_r(A) + \sR_m(A) \subset \sH(A)$. To see that equality holds, it will be shown that
\begin{equation}\label{sMdef}
 \dim \frac{\dom A +\ran A}{\sR_r(A)}=\sum_{k=0}^{m-1} C_k.
\end{equation}
To this end, observe the following identity
 \[
\begin{split}
 \dim \frac{\dom A +\ran A}{\sR_r(A)}
&=\dim\frac{\dom A +\ran A}{\ran A+\sR_r(A)} \\
&\hspace{-0.5cm}+\sum_{k=1}^{m-1}\dim\frac{\ran A^k +\sR_r(A)}{\ran A^{k+1}+\sR_r(A)}
+\dim\frac{\ran A^{m}+\sR_r(A)}{\sR_r(A)},
\end{split}
\]
where for the last term one has by Lemma~\ref{Caramelo} that
\[
    \dim\frac{\ran A^{m}+\sR_r(A)}{\sR_r(A)} = \dim\frac{\sR_r(A)}{\sR_r(A)} = 0.
\]
This gives \eqref{sMdef}. To see that the sum~\eqref{50cent} is direct, let $x\in \sR_r(A) \cap \sR_m(A)$, then $x$ is of the form as the left hand side of equation~\eqref{umum}. Since $x\in\sR_r(A)$ one further has that $[x] = 0 \in \sM_k(A)$ for all $0\le k\le m-1$. Then, similar to Step~3, it follows that $x=0$.

\medskip\noindent
\emph{Step 5:}
It will be shown that
\begin{equation}\label{sMdeffg}
\ran A \cap \sR_m(A)=\spn \setdef{ x_k^i}{1\le i\le C_k,\ 1\le k\le m-1}.
\end{equation}
It is clear by construction of $\sR_m(A)$ that the right-hand side is contained in the left hand side, so it suffices to show that
\[
    \dim \big( A \cap \sR_m(A)\big) = \sum_{k=1}^{m-1} C_k.
\]
By~\eqref{50cent} it follows that
\[
\dim \ran A = \dim \big(\ran A \cap \sR_r(A)\big) + \dim \big(\ran A \cap \sR_m(A)\big),
\]
so it remains to show that
\begin{equation}\label{rancomp}
 \dim \frac{\ran A}{\ran A \cap \sR_r(A)}= \sum_{k=1}^{m-1} C_k.
\end{equation}
To this end, observe that
\begin{equation}\label{sMdeff}
\dim \frac{\dom A +\ran A}{\sR_r(A)}=\dim \, \frac{\dom A +\ran A}{\ran A+\sR_r(A)}+ \dim \frac{\ran A +\sR_r (A)}{\sR_r(A)}.
\end{equation}
By~\eqref{sMdef} the left-hand side of the above equation equals $\sum_{k=0}^{m-1} C_k$ and the first term on the right-hand side is $C_0$. Hence,
\[
 \sum_{k=1}^{m-1} C_k = \dim \frac{\ran A +\sR_r (A)}{\sR_r(A)}=\dim \frac{\ran A}{\ran A \cap \sR_r(A)},
\]
where the last equality is due to~\cite[Lem~2.2]{Ka}. This proves \eqref{rancomp}.

\medskip\noindent
\emph{Step 6}:
For  \eqref{9Euro} it suffices to show that $A \subset A_R \hplus A_M$ and that the sum is direct.
Let $(x,y)\in A$, so that $y \in \ran A$. Then by~\eqref{50cent} one has
$y=y_r+y_m$ with $y_r\in \ran A \cap \sR_r(A)$ and $y_m\in \ran A \cap \sR_m(A)$.
Therefore, invoking~\eqref{sMdeffg}, there are $x_r \in\sR_r(A)$ and $x_m\in\sR_m(A)$ such that
$(x_r, y_r)\in A_R$ and $(x_m, y_m)\in A_M$. It follows that $(x-x_r-x_m, 0)\in A$, thus
$x-x_r-x_m\in\sR_0(A)$ and $(x-x_r-x_m, 0)\in A_R$. Therefore, one obtains that $(x-x_m,y-y_m)\in A_R$ and hence
\[
 (x,y)=(x-x_m, y-y_m) +(x_m,y_m) \in A_R \hplus A_M.
\]
That the sum \eqref{9Euro} is direct follows from \eqref{50cent}.

\medskip\noindent
\emph{Step 7}:
It will be shown that~\eqref{eq:ChainStrucAM} holds.
  It is clear that the right-hand side is contained in the left-hand side.
 For the converse inclusion,  let $(x,y)\in A \cap (\sR_m(A)\times \sR_m(A))$.
Since $y\in\ran A\cap \sR_m(A) $,  by \eqref{sMdeffg}
there is some $x_m\in\sR_m(A)$ with $(x_m, y)\in A_M$. It follows
that $x-x_m\in \sR_m(A)\cap\sR_0(A)=\{0\}$ by~\eqref{50cent}, thus $x=x_m$ and~\eqref{eq:ChainStrucAM} is shown.
It is a direct consequence of \eqref{eq:ChainStrucAM} that $A_M$ is an operator
(i.e., $\mul A_M=\{0\}$),
and that $\sigma_p(A_M) = \emptyset$, thus $A_M$ is a multishift.

\medskip\noindent
\emph{Step 8}: It remains to show~\eqref{eq:dim-AM}, which directly follows from~\eqref{Henk3}.
\end{proof}

The following is a consequence of Theorem \ref{Thm:DecompAR-AM}. It is implicitly contained in~\cite{SandDeSn05}.

 \begin{corollary}\label{multsh}
Let $A$ be a linear relation in a finite-dimensional linear space~$\sH$.
Then the following statements are equivalent:
\begin{enumerate}\def\labelenumi{\rm(\roman{enumi})}
\item $A$ is a multishift {\rm (}i.e., $\sigma_p(A) = \emptyset${\rm )};
\item there exists a basis for $\sH$ of the form \eqref{lilnasX} and $A$
is given by the right-hand side of \eqref {eq:ChainStrucAM}.
\end{enumerate}
\end{corollary}

%

%
%

\section{Main result: Jordan-like decomposition}\label{Sec:Decomp}

In this section the main result of this note will be stated.
Any linear relation in a finite-dimensional space admits a reducing sum decomposition
into a completely singular relation, a Jordan relation, and a multishift.
This is obtained by a combination of Theorems~\ref{Blaubeere},
\ref{Thm:DecompAS-AJ}, and \ref{Thm:DecompAR-AM}.
Furthermore, it will be shown that the chain structure of singular chains, Jordan chains,
and shift chains of any such decomposition is uniquely determined by~$A$ and given by its Weyr characteristics. Moreover, it turns out that the resulting decomposition of $A$ is a unique representative of the equivalence class of with respect to the notion of strict equivalence (see Definition~\ref{Def:strict-equi}).

\begin{theorem}\label{splitit}
 Let $A$ be a linear relation in a finite-dimensional space $\sH$.
Then there exist linear relations $A_S, J_{\la_1}(A), \ldots, J_{\la_l}(A), J_\infty(A),  A_M$, all contained in $A$,
where $\{\la_1, \ldots, \la_l\}=\si_\pi(A)\cap\dC$, such that
\begin{equation}\label{holythreefold}
A = A_S \oplus J_{\la_1}(A)\oplus \cdots\oplus J_{\la_l}(A)\oplus
J_\infty(A) \oplus A_M,
\end{equation}
is a reducing sum decomposition of $A$ with respect to
\begin{equation}\label{eq:holythreefold-spaces}
\sH(A)=\sR_c(A) \oplus \sX_{\la_1}(A) \oplus \cdots \oplus \sX_{\la_l}(A)\oplus \sX_\infty(A)\oplus\sR_m(A)
\end{equation}
with the spaces defined in Theorems~{\rm \ref{Blaubeere}},
{\rm \ref{Thm:DecompAS-AJ}}, and {\rm \ref{Thm:DecompAR-AM}}.
Furthermore,
\begin{enumerate}\def\labelenumi{\rm(\alph{enumi})}
  \item $A_S$ is completely singular in $\sR_c(A)$;
  \item $J_{\la_i}(A)$ is a Jordan operator in $\sX_{\la_i}(A)$ corresponding to $\la_i$ for $1\le i\le l$;
   \item $J_\infty(A)$ is a Jordan relation in $\sX_\infty(A)$; 
  \item $A_M$ is a multishift in $\sR_m(A)$.
\end{enumerate}
Any of the linear relations in~\eqref{holythreefold} may be absent, if the corresponding space in~\eqref{eq:holythreefold-spaces} is trivial.
\end{theorem}


%

\begin{remark}
The following special cases may serve to illustrate Theorem~\ref{splitit}.
\begin{enumerate}[(a)]
  \item Consider the case of a trivial singular chain subspace $\sR_{c}(A)=\{0\}$.
Then the completely singular part is absent and the treatment in Section~\ref{sec3} becomes
simpler.  In this case the quotient spaces $\sZ_k(A,\la)$  in \eqref{vv0}
are given by
\begin{equation}\label{jaca1}
\ker (A-\lambda),\,  \frac{\ker  (A-\lambda)^2}{\ker (A-\lambda)},
\, \frac{\ker  (A-\lambda)^3}{\ker (A-\lambda)^2}, \,\,\cdots ,
\end{equation}
whereas the quotient spaces $\sW_k(A)$  in \eqref{grijp0} are given by
\begin{equation}\label{jaca2}
\mul A,\,  \frac{\mul A^{2}}{\mul A},
\, \frac{\mul A^3}{\mul A^2}, \,\,\cdots .
\end{equation}
Recall that $\sR_{c}(A)=\{0\}$ implies that
$\sigma_{\pi}(A)=\sigma_{p}(A)$. Hence, if $\lambda \not\in \sigma_{p}(A)$, then
the Weyr characteristic corresponding to \eqref{jaca1}  is the null sequence and,
similarly, if $\infty \not\in \sigma_{p}(A)$, then
the Weyr characteristic corresponding to \eqref{jaca2}  is the null sequence.

\item Consider the case of a trivial multivalued part $\mul A=\{0\}$. Then certainly $\sR_{c}(A)=\{0\}$ (and the comments of~(a) apply), but additionally the quotient spaces in \eqref{jaca2} are trivial. The Weyr characteristic for \eqref{jaca1} then essentially coincides with that considered in~\cite{S15} for linear operators.

\item Consider the case of $\dom A = \sH$. Then $A_M$ equals the zero space. Assume that $A_M\neq \{(0,0)\}$. As \eqref{holythreefold}
is a reducing sum decomposition, one can assume, for simplicity,
$A=A_M$. As $A=A_M$ is a multishift, it has no eigenvalues and for every for every pair $(x,y) \in A$ the entries $x,y$
are linearly independent. Let $(x_1,x_2) \in A$. As $\dom A=\sH$,
it follows that $x_2\in \dom A$, hence there exists $x_3$
with $(x_2,x_3)\in A$. Now, $\{x_1, x_2\}$ is linearly independent but
$\{x_1, x_2, x_3\}$ might be linearly
independent or not. If it is linearly independent, then there exists
$x_4$ with $(x_3,x_4)\in A$. Again, $\{x_1, x_2, x_3, x_4\}$ is linearly
independent or not. If it is linearly independent, then  there exists
$x_5$ with $(x_4,x_5)\in A$. This can be continued. Finally, as
$\dom A = \sH$ is finite dimensional, this procedure shows
that there is a smallest natural number $m$, $2\leq m\leq \dim \sH$,
with the properties
$$
\begin{array}{ll}
\{x_1, \ldots, x_m\} & \mbox{ are linearly independent,}\\
\{x_1, \ldots, x_{m+1}\} & \mbox{ are not linearly independent,}\\
(x_i,x_{i+1} ) \in A & \mbox{ for } i=1,\ldots , m.
\end{array}
$$
Therefore, one has $x_{m+1}=\sum_{i=1}^m\alpha_ix_i$ for some
$\alpha_i\in \mathbb C$, $i=1,\ldots,m$. Set $M:=\spn\{x_1,\ldots,x_m\}$ and define the matrix
$$
T := 
\begin{bmatrix}
0 & \cdots &0& \alpha_1\\
1 & && \vdots\\
& \ddots && \vdots\\
&& 1 & \alpha_m
\end{bmatrix}.
$$
Let $\lambda\in\dC$ be an eigenvalue of  $T$ with eigenvector $\beta = (\beta_1,\ldots,\beta_m)^\top\in\dC^m\setminus\{0\}$, i.e., $T\beta = \lambda \beta$. Then $x:= \sum_{i=1}^m \beta_i x_i \in M$ is nontrivial and since $(x_i, x_{i+1}) \in A$ for $i=1,\ldots,m$ one finds that for
\[
    z:= \beta_1 x_2 + \ldots + \beta_{m-1} x_m + \beta_m x_{m+1}
\]
one has $(x, z) \in A$. Observe that by $x_{m+1}=\sum_{i=1}^m\alpha_ix_i$ it follows
\[
    z = \alpha_1 \beta_m x_1 + (\alpha_2 \beta_m + \beta_1) x_2 + \ldots + (\alpha_m \beta_m + \beta_{m-1}) x_m
\]
and since
\[
    \begin{pmatrix} \alpha_1 \beta_m \\ \alpha_2 \beta_m + \beta_1\\ \vdots \\ \alpha_m \beta_m + \beta_{m-1} \end{pmatrix} = T \begin{pmatrix} \beta_1 \\ \vdots \\ \beta_{m}\end{pmatrix} = \lambda \begin{pmatrix} \beta_1 \\ \vdots \\ \beta_{m}\end{pmatrix},
\]
one has that
\[
    z = \lambda \beta_1 x_1 + \ldots + \lambda \beta_m x_m = \lambda x
\]
so that $(x, \lambda x) \in A$ and hence $A$ has an eigenvalue, a contradiction.
Therefore, the quotient spaces in~\eqref{eq:Dk} are trivial, and the corresponding Weyr characteristic given by~\eqref{AdanYEva} is the null sequence. That is, the multishift part is absent.

\item As a consequence of (a)--(c), the classical result of the Jordan canonical form for linear operators $A$ in a finite-dimensional space $\sH$ is covered by Theorem~\ref{splitit}. Since in particular $\dom A = \sH$ and $\mul A = \{0\}$, the Weyr characteristic of $A$ is that corresponding to the spaces~\eqref{jaca1} and $A$ has the reducing sum decomposition
    \[
        A = J_{\la_1}(A)\oplus \cdots\oplus J_{\la_l}(A),
    \]
    with Jordan operators $J_{\la_i}(A)$ whose structure coincides with that  of classical Jordan blocks according to the representation~\eqref{coro}.
\end{enumerate}
\end{remark}

In order to justify calling~\eqref{holythreefold} a Jordan-like decomposition for linear relations it needs to exhibit a certain uniqueness. Recall that for the fixed decomposition of $\sH(A)$ in~\eqref{eq:holythreefold-spaces}, any reducing sum decomposition is intrinsically unique, cf.\ Section~\ref{Sec:preli}. Moreover, the Jordan-like decomposition is uniquely determined by the Weyr characteristic of~$A$; in particular, if any two linear relations have the same Weyr characteristic, then they have the same Jordan-like decomposition. To see this, recall that for a linear relation $A$ in a finite-dimensional linear space $\sH$
the Weyr characteristic corresponding to the sequence of quotient spaces
%
\begin{enumerate}\def\labelenumi{\rm(\alph{enumi})}
\item 
in \eqref{tunesandI} is given by the sequence $B:=(B_{k})_{k \geq 1}$ in \eqref{tunesandI+};
\item 
in \eqref{vv0} is given by the sequece $W(\lambda):=(W_{k}(\lambda))_{k \geq 1}$ in
\eqref{herni3};  
\item 
in \eqref{grijp0} is given by the sequence $A=(A_{k})_{k \geq 1}$ in \eqref{grijp1};
\item 
in \eqref{eq:Dk} is given by the sequence  $C:=(C_{k})_{k \geq 0}$
in \eqref{AdanYEva}.
\end{enumerate}   
Note that each of these Weyr characteristics is a finitely supported nonincreasing sequence,
which may be the null sequence.
The Weyr characteristics corresponding to all different proper complex eigenvalues $\{\la_1, \ldots, \la_l\}=\si_\pi(A)\cap\dC$
will be collected in a single sequence:
\begin{equation}\label{Ws}
W:=(W(\lambda_1), W(\lambda_2), \ldots, W(\lambda_{l})).
\end{equation}

\begin{definition}\label{wweyr}
Let $A$ be a linear relation in a finite-dimensional linear space $\sH$.
The collection of the sequences
\begin{equation}\label{Anitta}
(B, W,A,C),
\end{equation}
given by \eqref{tunesandI+},  \eqref{Ws}, \eqref{grijp1}, and \eqref{AdanYEva},
is called the \emph{Weyr characteristic} of the linear relation $A$.
\end{definition}

The Jordan-like decomposition~\eqref{holythreefold}  of a linear relation $A$ in Theorem~\ref{splitit} is completely determined by the Weyr characteristic of $A$, which follows from the construction in Theorems~\ref{Blaubeere}, \ref{Thm:DecompAS-AJ}, and \ref{Thm:DecompAR-AM}. Moreover, one has the following result.

\begin{proposition}\label{Prop:Jordan-Weyr}
 Any two linear relations in a finite-dimensional space $\sH$ with the same Weyr characteristic have the same reducing sum decomposition~\eqref{holythreefold} with respect to the same subspace decomposition~\eqref{eq:holythreefold-spaces}.
\end{proposition}

In the remainder of this section consider finitely supported nonincreasing sequences
\begin{equation}\label{eq:TheWire}
    \big( (B_k)_{k\ge 1}, (W^1_k)_{k\ge 1},\ldots, (W^l_k)_{k\ge 1}, (A_k)_{k\ge 1}, (C_k)_{k\ge 0}\big),
\end{equation}
where any of the sequences may be a null sequence. Then, by the above results, it is possible to construct a linear relation (given by the Jordan-like decomposition~\eqref{holythreefold}) which has~\eqref{eq:TheWire} as Weyr characteristic. But to which extent is this relation unique?  To answer this question one introduces the following notion of strict equivalence.

\begin{definition}\label{Def:strict-equi}
The linear relations $S_1$ and $S_2$
in a finite-dimensional space $\sH$  are said to be \emph{strictly equivalent}
if there exists an invertible matrix $T$ such that
\begin{align}\label{DuaLipa1}
  (x,y) \in S_2 \quad \iff \quad (T^{-1}x,T^{-1}y) \in S_1
\end{align}
or, what is the same,
\begin{align}\label{DuaLipa}
S_2=TS_1T^{-1}.
\end{align}
\end{definition}

Note that~\eqref{DuaLipa} is understood in the sense of multiplication of linear relations.
The following theorem is taken from \cite{GMPT}.

\begin{theorem}\label{notsonew}
Two linear relations in a  finite-dimensional space $\sH$
are strictly equivalent if and only if their Weyr
characteristics coincide.
\end{theorem}

As a direct consequence of Proposition~\ref{Prop:Jordan-Weyr} and Theorem~\ref{notsonew}, it follows that the Jordan-like decomposition~\eqref{holythreefold} is a unique representative of the equivalence classes with respect to strict equivalence. Furthermore, one has the following result.

\begin{theorem}\label{newnew}
For any given finitely supported nonincreasing sequences~\eqref{eq:TheWire} {\rm (}and a finite-dimensional space $\sH$ with sufficiently large dimension{\rm )} there exists, up to strict equivalence, exactly one linear relation $A$ in $\sH$ with Weyr characteristic~\eqref{eq:TheWire}.
\end{theorem}

\begin{remark}
The Jordan-like decomposition of linear relations derived in Theorem~\ref{splitit} resolves a ``non-uniqueness issue'' of the decomposition from~\cite{SandDeSn05}. A componentwise
direct sum decomposition of a linear relation $A$ into a completely singular relation,
a Jordan part and a multishift was derived in \cite{SandDeSn05}. However, this decomposition
does not exhibit uniqueness as the following example shows: For linearly independent elements $x_1, x_2, x_3$ in a finite-dimensional linear space $\sH$ define
\[
    A_1 := \spn \{ (0,x_1), (x_1,x_2), (x_2,0)\},\quad A_2 := \{(x_1,x_3)\}
\]
and $A:= A_1 \oplus A_2$.
Obviously, $\sR_c(A) = \spn\{ x_1, x_2\}$ and $\sR_r(A) = \spn\{ x_1, x_2, x_3\}$. Clearly, $A_1$ is completely singular and $A_2$ is a multishift, but $A_1 \oplus A_2$ is not a reducing sum decomposition. Furthermore, there is an alternative decomposition of $A$ into
\[
    A = A_1 \oplus A_3,\quad A_3 = \spn\{ (0,x_3 - x_2)\},
\]
where $A_3$ consists of a Jordan chain at $\infty$. Both decompositions are possible in the framework of \cite{SandDeSn05}. On the other hand, the Jordan-like decomposition~\eqref{holythreefold} in Theorem~\ref{splitit} is a reducing sum decomposition, and hence unique for the fixed decomposition of $\sH(A)$ in~\eqref{eq:holythreefold-spaces}.

It should also be stressed that in the present paper
the reducing sum decompositions are derived in the setting of linear spaces;
no further structure (such as an inner product) is required.
\end{remark}

\begin{remark}
The presentation of some of the material in \cite{SandDeSn05} was inspired
by the results in \cite{KW1, KW2}.
In the setting of (what is now called) almost Pontryagin spaces,
Kaltenb\"ack and Woracek  considered selfadjoint extensions of
symmetric relations with defect numbers $(1,1)$; one of the
requirements was the existence of a shift chain relative to
the isotropic part of the almost Pontryagin space.
The multishifts in \cite{SandDeSn05} were introduced
with the work of  Kaltenb\"ack and Woracek in mind.
Shifts have also been considered in the context of Pontryagin spaces;
see for instance \cite{CD20}, where
references to further work can be found.
\end{remark}

\begin{remark}
 Let $E,F\in\mathbb C^{n\times m}$ be matrices and let $s E-F$
be the corresponding matrix pencil.
Associated with $E$ and $F$  are two linear relations
\begin{equation}\label{inverses2}
E^{-1} F = \setdef{(x,y)\in \dC^m\times\dC^m }{ Fx=Ey}
\quad \mbox{and} \quad
FE^{-1}= \setdef{ (Ex,Fx) }{ x\in \mathbb C^m},
\end{equation}
which were already studied in~\cite{BennByer01,BB06}.
Usually, $E^{-1} F$ is called the {\em kernel} representation and $FE^{-1}$ the {\em range} representation
(see also \cite{BTW16}). Matrix pencils have
a canonical form, the so-called Kronecker canonical form~\cite{BergTren12,G59,K90}.
{ There is a deep connection between the range and the kernel representation and the corresponding matrix pencil. This was already
utilized in \cite{BGTWW,GMPST,GT17,GT,LMPWT}.}
A complete set of invariants for
the Kronecker canonical form are four multi-indices: the finite and
 infinite elementary divisors, the column and the row minimal indices.
These quantities measure the sizes of the different blocks in the
Kronecker canonical form.
Moreover, two matrix pencils are strictly equivalent if and only if all the four indices coincide \cite{G59}. They are completely determined by the so-called Wong sequences~\cite{BergTren13}, which are certain sequences of subspaces; the geometric approach in~\cite{BergTren13} is, in its spirit, close to the approach in the present paper (although not quotient spaces have been used). The relationship between linear relations and matrix pencils is investigated in \cite{BTW16,GMPT} and it
will be continued
in upcoming papers.
 \end{remark}
}

\end{document}